\documentclass[12pt]{article}
\usepackage{amsfonts}
\usepackage{amssymb}
\usepackage{amsmath}
\usepackage{amsthm}
\usepackage{verbatim}

\newtheorem{thm}{Theorem}[section]
\newtheorem{prop}[thm]{Proposition}

\newtheorem{cor}[thm]{Corollary}

\theoremstyle{definition} 

\newtheorem{defn}[thm]{Definition}
\newtheorem{exa}[thm]{Example}

\newcommand{\m}{\medskip}

\newcommand{\e}{\varepsilon}

\newcommand{\Om}{\Omega}
\newcommand{\om}{\omega}

\newcommand{\R}{\mathbb{R}}
\newcommand{\Z}{\mathbb{Z}}

\newcommand{\T}{\mathbb{T}}

\newcommand{\D}{\mathcal{D}}
\newcommand{\F}{\mathcal{F}}

\newcommand{\Dd}{\D^\delta}

\newcommand{\bmo}{{\rm BMO}}
\newcommand{\vmo}{{\rm VMO}}

\newcommand{\intav}{-\!\!\!\!\!\!\int}
\newcommand{\intq}{\frac{1}{|Q|} \, \int_Q \,}

\DeclareMathOperator{\supp}{supp}
\DeclareMathOperator*{\essinf}{ess\ inf}
\DeclareMathOperator*{\esssup}{ess\ sup}

\def\XXint#1#2#3{{\setbox0=\hbox{$#1{#2#3}{\int}$}
     \vcenter{\hbox{$#2#3$}}\kern-.5\wd0}}

\allowdisplaybreaks

\begin{document}

\title{One-parameter and multiparameter function classes\\
    are intersections of finitely many dyadic
    classes.}

\renewcommand{\thefootnote}{\fnsymbol{footnote}}

\author{Ji Li\footnote{Research supported by a University of
South Australia postdoctoral fellowship, and by the NNSF of China,
Grant No.~11001275.}\\
Department of Mathematics\\
Sun Yat-Sen University\\
Guangzhou, PRC\\
{\tt liji6@mail.sysu.edu.cn}
\and Jill Pipher\footnote{Research
supported by the NSF under grant
    number DMS0901139.}\\
Department of Mathematics\\
Brown University\\
Providence, RI 02912, USA\\
{\tt jpipher@math.brown.edu}
\and Lesley A.~Ward\\
School of Mathematics and Statistics\\
University of South Australia\\
Mawson Lakes SA 5095, Australia\\
{\tt lesley.ward@unisa.edu.au}}

\renewcommand{\thefootnote}{\arabic{footnote}}


\maketitle

\date

\vspace{-.3cm}

\begin{abstract}
    We prove that the class of Muckenhoupt $A_p$ weights
    coincides with the intersection of finitely many suitable
    translates of dyadic~$A_p$, in both the one-parameter and
    multiparameter cases, and that the analogous results hold
    for the reverse H\"older class~$RH_p$, for doubling
    measures, and for the space~$\vmo$ of functions of
    vanishing mean oscillation. We extend to the multiparameter
    (product) space~$\bmo$ of functions of bounded mean
    oscillation the corresponding one-parameter $\bmo$ result
    due to T.~Mei, by means of the Carleson-measure
    characterization of multiparameter~$\bmo$. Our results hold
    in both the compact and non-compact cases. In addition, we
    survey several definitions of \vmo\ and prove their
    equivalences, in the continuous, dyadic, one-parameter and
    multiparameter cases. We show that the weighted Hardy space
    $H^1(\om)$ is the sum of finitely many suitable translates
    of dyadic weighted~$H^1(\om)$, and that the weighted
    maximal function is pointwise comparable to the sum of
    finitely many dyadic weighted maximal functions for
    suitable translates of the dyadic grid and for each
    doubling weight~$\om$.
\end{abstract}


\section{Introduction}
\label{sec:introduction}
\setcounter{equation}{0}

Function spaces and function classes are of considerable
interest in harmonic analysis, since (i) the prototypical
problem is to establish the boundedness of a singular integral
operator from one function space to another, (ii) these
operators also act on $L^p$ spaces weighted by functions in the
$A_p$ or $RH_p$ function classes, and (iii) the density of the
measure on the underlying space is commonly assumed to belong
to the class of doubling weights. Dyadic function classes offer
a parallel setting in which calculation is often simpler, since
one can exploit the geometry of the dyadic intervals. For
example, in~\cite{JoNi} the John--Nirenberg inequality is
proved by establishing a related inequality on certain dyadic
cubes arising from a Calder\'on--Zygmund decomposition. In the
current paper we are concerned with a \emph{bridge between
continuous and dyadic function classes}.

We consider the \emph{function spaces} \bmo, \vmo\ and $H^1$
(respectively of functions of bounded mean oscillation,
functions of vanishing mean oscillation, and the Hardy space),
and also the \emph{weight classes} of Muckenhoupt's $A_p$
weights, of reverse H\"older $RH_p$ weights, and of doubling
weights. We use the term \emph{function classes} to include
both function spaces and weight classes. We show that for each
of these function classes, the continuous version can be
written as the intersection of finitely many suitable
translates of the dyadic version, using two translates in the
one-parameter case, and $2^k$ translates in the multiparameter
case. We also show how the norms (or the $A_p$ and $RH_p$
constants) from the continuous and dyadic classes are related.
Our results hold both in the compact case where our functions
are defined on the circle~$\T$ and in the non-compact case of
the real line~$\R$, and for the higher-dimensional
one-parameter analogues $\T^m$ and~$\R^m$. As corollaries, we
extend to the weighted case and to the multiparameter case the
result that the Hardy space $H^1$ is the sum of finitely many
suitable translates of the dyadic version of~$H^1$. Also, we
show that the weighted maximal function, where the weight is
doubling, is pointwise comparable to the sum of finitely many
suitable translates of the dyadic weighted maximal functions.

We note that Mei established these results for the case of
one-parameter~\bmo, for one-parameter unweighted $H^1$, and for
the one-parameter unweighted maximal function~\cite{Mei}.

We mention that there is a second type of bridge between
continuous and dyadic function classes, via averaging, as
developed in the papers~\cite{GJ,W,PW,T,PWX}. Namely, a
suitable family of functions in the dyadic version of a
function class can be converted to a single function that
belongs to the continuous version of the same class, via a
translation-average (for \bmo\ and \vmo) or a
geometric-arithmetic average (for $A_p$, $RH_p$, and doubling
weights). We do not discuss these matters further in the
current paper.

What is a ``suitable translate''? Denote by $\D$ the usual grid
of dyadic intervals in~$\R$. Consider the real numbers $\delta$
that are \emph{far from the dyadic rational numbers}, in the
sense that the distance from $\delta$ to each given dyadic
rational~$k/2^n$ is at least some fixed multiple of $1/2^n$.
That is,
\begin{equation}\label{eqn:badlyapproximable}
    \left|\delta - \frac{k}{2^n}\right|
    \geq \frac{C}{2^n}
    \qquad \text{for all integers $n$ and $k$},
\end{equation}
where $C$ is a positive constant that may depend on $\delta$
but is independent of~$n$ and~$k$. For example, $\delta = 1/3$
is far from the dyadic rationals. The set of all such $\delta$ is
dense in~$\R$ but has measure zero~\cite{Mei}.

For such a $\delta$, let $\Dd$ denote the translate of $\D$ by
$\delta$, modified as follows. Small dyadic intervals are
simply translated by~$\delta$. Large dyadic intervals are
translated not only by $\delta$ but also by an additional
amount which depends on the scale of the interval. See
Section~\ref{sec:dyadicintervals} for the precise construction
of this collection $\Dd$ of translated dyadic intervals and for
a motivating example.


The key proposition in Mei's paper (\cite[Prop 2.1]{Mei})
states that if~$\delta$ is far from the dyadic rationals, then
for each interval~$Q$ there is an interval~$I$ containing~$Q$,
belonging either to the grid~$\D$ of dyadic intervals or to the
grid~$\Dd$ of translated dyadic intervals, and whose length~$I$
is no more than~$C|Q|$, where~$C$ is a constant independent
of~$Q$.

We observe in Proposition \ref{prop:comparable} below that if
$\delta$ is far from the dyadic rationals and if a weight~$\om$
is both dyadic $\D$-doubling and  dyadic $\Dd$-doubling, then
for each interval~$Q$, the average of~$\om$ on~$|Q|$ is
comparable to the average of~$\om$ on the interval~$I$
guaranteed by \cite[Prop 2.1]{Mei}, with constants independent
of~$Q$. We use this observation in a crucial way in the proofs
of our results for $A_p$, $RH_p$, $H^1(\om)$ and weighted
maximal functions; see below.

We note that Mei's key proposition is a generalization of the
so-called one-third trick. The earliest reference we have found
for this idea is in \cite[p.339]{O}, although it may have been
known earlier.

If a nonnegative locally integrable function $\om$ belongs to
the class of $A_p$ weights, for some $p$ with $1 \leq p \le
\infty$, then \emph{a fortiori} $\om$ belongs to the class
$A_p^d$ of dyadic $A_p$ weights, for which the defining $A_p$
condition is only required to hold for dyadic
intervals~($I\in\D$), not for all possible intervals. Similarly
$\om$ belongs to the class $A_p^{\delta}$, for which the $A_p$
condition is only required to hold for appropriate translates
of the dyadic intervals~($I\in\Dd$).  Thus $A_p \subset
A_p^d\cap A_p^{\delta}$.

We show in this paper that if $\delta$ is far from dyadic
rationals in the sense of
condition~\eqref{eqn:badlyapproximable}, then equality holds:
$A_p = A_p^d \cap A_p^{\delta}$. Moreover, the
constant~$A_p(\om)$ depends only on the constants $A_p^d(\om)$
and~$A_p^\delta(\om)$, and vice versa.

We also establish the corresponding result for reverse H\"older
weights: if $\delta$ satisfies
condition~\eqref{eqn:badlyapproximable} and $1 \leq p \leq
\infty$, then $RH_p = RH_p^d \cap RH_p^{\delta}$, with the
corresponding dependence of the $RH_p$ constants. Further, the
same is true for doubling weights. These results for $A_p$, for
$RH_p$ and for doubling weights are collected in Theorem
\ref{thm:intersectionoftranslatesApRHpdbl} for the
one-parameter case, and Theorem \ref{thm:product
intersectionoftranslatesApRHpdbl} for the multiparameter case
with arbitrarily many factors.

Next we point out that T.~Mei~\cite{Mei} established the
corresponding result for \bmo\ on the circle~$\T$: if $\delta$
is far from dyadic rationals in the sense of
condition~\eqref{eqn:meicondR}, then $\bmo(\T) = \bmo_d(\T)
\cap \bmo_{\delta}(\T)$. He extended this result to
(one-parameter)~$\T^m$, showing that $\bmo(\T^m)$ is the
intersection of $m + 1$ translates of the dyadic version of
$\bmo(\T^m)$, and similarly to (one-parameter)~$\R^m$. He notes
that John Garnett knew earlier that $\bmo$ coincides with the
intersection of three translates of dyadic $\bmo$, building
on~\cite[p.417]{Gar}.

We give an alternative proof of Mei's $\bmo$ result, via the
Carleson-measure characterization of \bmo\ together with Mei's
key proposition, in order to generalize to the case of
multiparameter \bmo; see Theorems
\ref{thm:intersectionoftranslatesBMOR}
and~\ref{thm:intersectionoftranslatesBMORn}.

We also prove the analogous results for one-parameter and
multiparameter \vmo\ (Theorem \ref{thm:intersectionoftranslates
VMO}). In doing so, we survey several definitions of continuous
and dyadic \vmo\ in the one- and multiparameter settings, the
equivalences among these definitions, and the duality $\vmo^* =
H^1$.

As a consequence of the results above, we show that the
weighted Hardy space $H^1(\om)$ is the sum of finitely many
suitable translates of dyadic weighted~$H^1(\om)$, and show
that the weighted maximal function (Hardy--Littlewood maximal
function or strong maximal function, as appropriate) is
pointwise comparable to the sum of finitely many dyadic
weighted maximal functions for suitable translates of the
dyadic grid and for every doubling weight~$\om$ (in particular
for weights $w\in A_p$, $1\leq p\leq \infty$). See Proposition
\ref{prop:H-L function and Hardy space} and Corollary
\ref{cor:weighted H-L function and Hardy space} for the
one-parameter case and Propositions
\ref{prop:sum_of_translates_productH1},
\ref{prop:sum_of_translates_strong maximal} and Corollary
\ref{cor:product weighted H-L function and Hardy space} for the
multiparameter case.

We make some remarks comparing the compact case (the
circle~$\T$) with the non-compact case (the real line~$\R$). We
define the circle to be the unit interval with endpoints
identified: $\T := [0,1]/(0\sim 1)$. Note that non-dyadic
intervals $Q\subset\T$ may wrap around from 1 to~0. First, for
the continuous function space \bmo\ and the continuous function
classes of $A_p$, $RH_p$ and doubling weights, there is only a
small difference between the compact and non-compact cases: the
defining property is assumed to hold only on the intervals
contained in~$\T$ as opposed to on all intervals in~$\R$.
Second, for their dyadic versions ($\bmo_d$, $A_p^d$, $RH_p^d$
and dyadic doubling weights), the same is true, with the
additional difference that when considering translations
by~$\delta$, in the compact case~($\T$) it suffices simply to
translate each dyadic interval by~$\delta$, while in the
non-compact case~$\R$, we translate intervals of length larger
than 1 not only by the amount~$\delta$ but also by an
additional amount that depends on the scale, as mentioned
above.

The differences in the case of \vmo\ are more subtle. First,
for continuous \vmo, in the compact case the definition of the
subspace \vmo\ of \bmo\ involves a condition requiring the
oscillation of the function to approach zero as the length of
the interval goes to zero. In the non-compact case, one must
impose two additional conditions controlling the oscillation
over large intervals and over intervals that are far from the
origin. (With this definition one retains the duality $\vmo^* =
H^1$.) Second, for the dyadic non-compact case the same three
oscillation conditions apply, and also when translating by
$\delta$ we need the additional translations of intervals at
large scales, as described in the preceding paragraph. See
Section~\ref{sec:resultsproductVMO} for a detailed discussion
of \vmo, including references.

For the one-parameter but higher-dimensional cases, where the
functions are defined on $\T^m$ or $\R^m$, $m \geq 2$, instead
of on~$\T$ or~$\R$, the same remarks apply, with intervals
replaced by cubes.

So much for the one-parameter case. The corresponding remarks
apply to the multiparameter case. We give the technical details
in the body of the paper.

To reduce the amount of notation required, in the rest of the
paper we work on~$\R$ and $\R\otimes\R$. However, our results
and proofs go through for $\T$ and $\T\otimes\T$, and for
$\R^m$, $\R^{m_1}\otimes\R^{m_2}$, $\T^m$, and
$\T^{m_1}\otimes\T^{m_2}$, and also for arbitrarily many
factors in the multiparameter setting ($\R^{m_1}\otimes\cdots
\otimes\R^{m_k}$ or $\T^{m_1}\otimes\cdots \otimes\T^{m_k}$).

The paper is organized as follows. In Section
\ref{sec:dyadicintervals} we give the required background on
grids of dyadic intervals, their translates,  real numbers
$\delta$ that are far from dyadic rationals, and the key
proposition in Mei's paper. In
Section~\ref{sec:resultsApRHpdbl} we define the classes of
$A_p$ weights, $RH_p$ weights and doubling weights, including
the extreme cases $A_1$, $A_\infty$, $RH_1$ and $RH_\infty$. We
prove our results for all these classes in both the
one-parameter and multiparameter cases. In
Section~\ref{sec:resultsBMO} we re-prove Mei's \bmo\ result,
and extend it to the multiparameter case. In
Section~\ref{sec:resultsproductVMO}, we discuss the definitions
of \vmo\ and prove the \vmo\ results. In
Section~\ref{sec:hardymaximal}, we prove our results for the
Hardy space~$H^1$ in both the one- and multiparameter cases,
and for the Hardy--Littlewood maximal function and the strong
maximal function, and we establish the weighted versions of
these results.


\section{Dyadic intervals and their translates}
\label{sec:dyadicintervals}
\setcounter{equation}{0}

Let $\D = \D(\R)$ denote the grid of dyadic intervals on $\R$:
$$
    \D(\R)=\bigcup_{n\in\mathbb{Z}} \D_n(\R),
$$
where for each $n\in\mathbb{Z}$,
$$
    \D_n(\R)
    = \Big\{ \Big[{k\over 2^n}, {k+1\over 2^n}\Big)\ \Big|\ k\in\Z \Big\}.
$$

On the circle $\T=[0,1]\slash(0,~1)$, the definition of $\D(\T)$ is
the same except that the integer $n$ runs only from 0 to $\infty$,
and $k\in \{0,1,\ldots,2^n-1\}$.

For $\delta\in\R$, we denote by $\D^\delta=\D^\delta(\T)$ the
translate to the right by $\delta$ of the dyadic grid on the circle
$\T$, considered modulo 1. Thus
$$
    \Dd(\T) := \{I + \delta \big|\ I\in\D(\T)\}
    = \Big\{ \Big[\big({k\over 2^n}+\delta\big)\ \text{mod}\ 1,
        \big({k+1\over 2^n}+\delta\big)\ \text{mod}\ 1\Big)\ \Big|\ k\in\Z \Big\}.
$$

Finally, on the real line $\R$, following Mei, we include
additional translates in the definition of the large-scale
$\delta$-dyadic intervals in $\D^\delta(\R)$. Specifically,
$$
   \D^\delta(\R)=\bigcup_{n\in\Z} D^\delta_n(\R),
$$
where for $n\geq0$,
$$
    D^\delta_n(\R)=\{I+\delta \mid I\in \D_n(\R)\},
$$
while for $n<0$ and $n$ even,
$$
    D^\delta_n(\R)
    =\Big\{\Big[{k\over 2^n}+\delta+\sum_{j=(n/2)+1}^0 2^{-2j},\ {k+1\over 2^n}
    +\delta+\sum_{j=(n/2)+1}^0 2^{-2j}\Big)\ \Big|\ k\in
    \Z\Big\}.
$$
These choices together with the nested property completely
determine the collections $\D^\delta_n(\R)$ for $n<0$, $n$ odd.
Note that we have translated by $\sum_{j=(n/2)+1}^0 2^{-2j}$,
rather than by Mei's translation $\sum_{j=n+2}^0 2^{-j}$.

For example, for $n=-2$, the interval $[0,4)$ of length $2^2$
belongs to $\D_{-2}(\R)$ while its translate $[\delta+1,\delta+5)$
belongs to $\D_{-2}^\delta(\R)$, since $\sum_{j=0}^02^{-2j}=1$.
Similarly, for $n=-4$, the interval $[0,16)$ of length $2^4$ belongs
to $\D_{-4}(\R)$ while its translate $[\delta+5,\delta+21)$ belongs
to $\D_{-4}^\delta(\R)$, since $\sum_{j=-1}^02^{-2j}=5$.

\begin{defn}\label{def: delta}
A real number $\delta$ is \emph{far from the dyadic rational
numbers} if the distance from $\delta$ to each given dyadic
rational~$k/2^n$ is at least some fixed multiple of $1/2^n$. That
is, if
\begin{equation*}
    \Big|\delta - \frac{k}{2^n}\Big|
    \geq \frac{C}{2^n}
    \qquad \text{for all integers $n$ and $k$},
\end{equation*}
where $C$ is a positive constant that may depend on $\delta$
but is independent of~$n$ and~$k$. Equivalently, the
\emph{relative distance} $d(\delta)$ from $\delta$ to the set
of dyadic rational numbers is positive:
\begin{equation}\label{eqn:meicondR}
    d(\delta)
    := \inf \Big\{2^n \Big|\delta -\frac{k}{2^n}\Big|
        \mid n \in \Z, k\in\Z\Big\}
    > 0.
\end{equation}
\end{defn}

For example, $\delta = 1/3$ is far from the dyadic rationals
since $d(1/3) = 1/3 > 0$. The set of all such $\delta$ is dense
in~$\R$ but has measure zero. Note that $d(\delta+1)=d(\delta)$
for all $\delta\in\R$.

The key tool in Mei's paper is the following proposition.

\begin{prop}[Prop 2.1\ \cite{Mei}]\label{prop:Mei}
    Suppose $\delta\in (0,1)$ is far from dyadic
    rationals, in the sense of condition~\eqref{eqn:meicondR}.
    Then there is a constant $C(\delta)$ such that for each
    interval~$Q$ in~$\R$, there is an interval~$I$ in~$\R$ such
    that
    \begin{enumerate}
        \item[\textup{(i)}] $Q\subset I$,

        \item[\textup{(ii)}] $|I| \leq C(\delta) |Q|$,
            and

        \item[\textup{(iii)}] $I\in\D$ or $I\in\Dd$.
    \end{enumerate}
    The constant~$C(\delta)$ can be taken to be $C(\delta) =
    2/d(\delta)$.
\end{prop}

Mei's proposition is stated on the circle~$\T$ identified with
$(0,2\pi]$, with condition~\eqref{eqn:meicondR} replaced by
    \begin{equation}\label{eqn:meicondT}
        d(\delta)
        := \inf \{2^n \left|\delta - k 2^{-n}\right|
            \mid n \geq 0, k\in\Z\}
        > 0,
    \end{equation}
with $0 < \delta < 1$ and with the filtrations $\D(\T)$ and
$\D^\delta(\T)$. For completeness, we give a proof, stated on
the real line and following Mei's proof.

\begin{proof}[Proof of Proposition~\ref{prop:Mei}]
Fix an interval~$Q$ in~$\R$. There exists an integer $n$ such
that $d(\delta)2^{-n-1}\leq |Q|< d(\delta)2^{-n}$. Now we set
$A_n=\{ k\cdot 2^n \mid k\in\Z \}$ for every fixed $n$ and
\begin{enumerate}
    \item[(1)]$A_n^\delta=\{ \delta+ k\cdot 2^n \mid k\in\Z
        \}$ for $n\geq 0$,

    \item[(2)]$A_n^\delta=\{ \delta+
        \sum_{j=(n/2)+1}^{0}2^{-2j}+ k\cdot 2^n \mid k\in\Z
        \}$ for $n<0$, $n$ even, and

    \item[(3)]$A_n^\delta=\{ \delta+
        \sum_{j=(n-1)/2+1}^{0}2^{-2j}+ k\cdot 2^n \mid
        k\in\Z \}$ for $n<0$, $n$ odd.
\end{enumerate}
Note that for any two different points $a,b\in A_n\cup
A_n^\delta$, we have $|a-b|\geq d(\delta)2^{-n}>|Q|$. Thus,
there is at most one element of $A_n\cup A_n^\delta$ belongs to
$Q$. Hence, $Q\cap A_n=\emptyset $ or $Q\cap
A_n^\delta=\emptyset $. Therefore, $I$ must be contained in
some dyadic interval $I\in \D$ or $I\in \D^\delta$ and
$|I|=2^{-n}\leq (2/d(\delta)) |Q|$.
\end{proof}

The corresponding result holds for intervals~$I'$ contained in~$Q$.

\begin{prop}
    Suppose $\delta\in\R$ is far from dyadic
    rationals, in the sense of condition~\eqref{eqn:meicondR}.
    Then there is a constant $C'(\delta)$ such that for each
    interval~$Q$ in~$\R$, there is an interval~$I'$ in~$\R$
    such that
    \begin{enumerate}
        \item[\textup{(i)}] $Q\supset I'$,

        \item[\textup{(ii)}] $|I'| \geq C'(\delta) |Q|$,
            and

        \item[\textup{(iii)}] $I'\in\D$ or $I'\in\Dd$.
    \end{enumerate}
\end{prop}

\begin{proof}
    For every interval $Q$ in $\R$, take the integer~$n$ such that
    $2^{-n}\leq|Q|<2^{-n+1}$. Then, there exists an interval $I$ in
    $\D_{-n-1}$ or $\D_{-n-1}^\delta$ such that $I\subset Q$. Hence
    the Proposition holds.
\end{proof}

\begin{exa}\label{exa:whyshift}
Here is an example that illustrates the difference between $\R$ and
$\T$ with regard to translations, and the need in the definition
above of the $\delta$-dyadic intervals in $\D^\delta(\R)$ for the global translations
at certain scales. On the real line~$\R$, take the usual
collection~$\D$ of dyadic intervals, take any positive $\delta$ and
let $\D^{\delta,\#}$ denote the translation to the right by $\delta$
of the dyadic grid~$\D$, so that $I^\#\in\D^{\delta,\#}$ if and only
if $I^\# = I + \delta$ for some $I\in\D$. Let $Q$ be an interval
containing both 0 and~$\delta$ in its interior. Then there is no
interval $I$ in $\D$ or $\D^{\delta,\#}$ satisfying property~(i),
namely $Q\subset I$, of Proposition~\ref{prop:Mei}, since dyadic
intervals do not have 0 as an interior point and intervals
in~$\D^{\delta,\#}$ do not have $\delta$ as an interior point.

On the circle~$\T$ viewed as $[0,1]$ with the endpoints identified,
however, the situation is different. Take $\delta
> 0$ such that $d(\delta) > 0$. It follows from the definition
of $\delta$, by taking $k = n = 0$, that $d(\delta) \leq \delta$.
Let $Q$ be an interval in~$\T$ containing 0 and $\delta$ as interior
points. Then $|Q| \geq \delta$. It follows that the properties
asserted in Proposition~\ref{prop:Mei} hold for the choice $I = \T$,
since $Q\subset\T$, $\T\in\D$, and $|\T| = 1 \leq |Q|/\delta$ so
that
\[
    \frac{|\T|}{|I_{t,y}|}
    \leq \frac{1}{\delta}
    \leq \frac{1}{d(\delta)}
    \leq \frac{2}{d(\delta)}
    = C(\delta).
\]


Mei's use of the \emph{translates} of the usual dyadic
intervals at scale $2^{-n}$ for all even $n<0$ \cite[Remark
7]{Mei} ensures that the conclusion of Proposition
\ref{prop:comparable} does hold for the intervals $Q\subset\R$
in this example.
\end{exa}

For a function~$\om$ that is both dyadic $\D$-doubling and dyadic
$\D^\delta$-doubling, where $\delta$ is far from dyadic rationals,
the average value of $\om$ over an interval~$Q$ is comparable to the
average value of $\om$ over the interval~$I$ guaranteed by
Proposition~\ref{prop:Mei}. See Proposition \ref{prop:comparable} below.

\section{Intersections of dyadic classes of weights}
\label{sec:resultsApRHpdbl}
\setcounter{equation}{0}

\subsection{One-parameter results for $A_p$ and $RH_p$}

We prove that $A_p$ is the intersection of two suitable
translates of dyadic $A_p$ for $1\leq p\leq \infty$, and that
the corresponding results hold for $RH_p$, $1\leq p\leq \infty$
and for doubling weights.

As usual, by \emph{doubling weight} we mean a nonnegative
locally integrable function $\om$ on $\R$ such that
$\om(\widetilde{Q})\leq C\om(Q)$ with a positive constant $C$
independent of $Q$, where the double $\widetilde{Q}$ of $Q$ is
the interval with the same midpoint as $Q$ and twice the length
of $Q$. Similarly, a \emph{dyadic doubling weight} satisfies
the corresponding condition $\om(\widetilde{I})\leq C\om(I)$,
where $\widetilde{I}$ is the dyadic parent of $I$.

The $A_p$ weights were identified by Muckenhoupt~\cite{M} as
the weights~$\om$ for which the Hardy--Littlewood maximal
function is bounded from~$L^p(d\mu)$ to itself, where $d\mu =
\om(x) \, dx$. Here we define the classes~$A_p$ and $RH_p$ in
the one-parameter setting. We delay the corresponding
definitions, statements and proofs for the multiparameter
setting until Subsection~\ref{subsec:productresultsApRHpdbl}.

We use the notation $\intav_Ef := {1\over |E|}\int_Ef$.
\begin{defn}
  \label{def:Ap}
  Let $\omega(x)$ be a nonnegative locally integrable function
  on~$\R$. For real~$p$ with $1 < p < \infty$, we
  say $\omega$ is an $A_p$ \emph{weight}, written $\omega\in
  A_p$, if
  \[
    A_p(\om)
    := \sup_Q \left(\intav_Q \omega\right)
    \left(\intav_Q
      \left(\dfrac{1}{\omega}\right)^{1/(p-1)}\right)^{p-1}
    < \infty.
  \]
  For $p = 1$, we say $\omega$ is an $A_1$ \emph{weight},
  written $\omega\in A_1$, if
  \[
    A_1(\om)
    := \sup_Q \left(\intav_Q \omega\right)
    \left(\dfrac{1}{\essinf_{x\in Q} \omega (x)}\right)
    < \infty.
  \]
  For $p = \infty$, we say $\omega$ is an \emph{$A$-infinity
  weight}, written $\omega\in A_\infty$, if
  \[
    A_\infty(\om)
    := \sup_Q \left(\intav_Q \omega\right)
    \exp\left(\intav_Q \log \left(\frac{1}{\omega}\right) \right)
    < \infty.
  \]
  Here the suprema are taken over all intervals~$Q\subset\R$.
  The quantity $A_p(\om)$ is called the \emph{$A_p$~constant
  of~$\om$}.

  The \emph{dyadic $A_p$ classes} $A_p^d$ for $1 \le p \le
  \infty$ are defined analogously, with the
  suprema~$A_p^d(\om)$ being taken over only the dyadic
  intervals~$I\subset\R$.
\end{defn}

\begin{defn}
  \label{def:RHp}
  Let $\omega(x)$ be a nonnegative locally integrable function
  on~$\R$. For real~$p$ with $1 < p < \infty$, we
  say $\omega$ is a \emph{reverse-H\"older-$p$ weight}, written
  $\omega\in RH_p$ or $\om\in B_p$, if
  \[
    RH_p(\om)
    := \sup_Q \left(\intav_Q \omega^p\right)^{1/p}
    \left(\intav_Q \omega\right)^{-1}
    < \infty.
  \]
  For $p = 1$,
  we say $\omega$ is a
  \emph{reverse-H\"older-1 weight}, written $\omega\in
  RH_1$ or $\om\in B_1$, if
  \[
    RH_1(\om)
    := \sup_Q \intav_Q\left(\frac{\om}{\intav_Q\om} \log\frac{\om}{\intav_Q\om}\right)
    < \infty.
  \]
  For $p = \infty$, we say $\omega$ is a
  \emph{reverse-H\"older-infinity weight}, written $\omega\in
  RH_\infty$ or $\om\in B_\infty$, if
  \[
    RH_\infty(\om)
    := \sup_Q \left(\esssup_{x\in Q} \om\right)
    \left(\intav_Q \omega\right)^{-1}
    < \infty.
  \]
  Here the suprema are taken over all intervals~$Q\subset\R$.
  The quantity $RH_p(\om)$ is called the \emph{$RH_p$~constant}
  of~$\om$.

  For $1 \leq p \leq \infty$, we say $\omega$ is a \emph{dyadic
  reverse-H\"older-$p$ weight}, written $\om\in RH_p^d$ or $\om
  \in B_p^d$, if
  \begin{enumerate}
  \item[(i)] the analogous condition $RH_p^d(\om)<\infty$
      holds with the supremum
      being taken over only the dyadic intervals
      $I\subset\R$, and

  \item[(ii)] in addition $\om$ is a dyadic doubling
      weight.
  \end{enumerate}
  We define the \emph{$RH_p^d$~constant} $RH_p^d(\om)$~of~$\om$
  to be the larger of this dyadic supremum and the dyadic
  doubling constant.
\end{defn}

The $A_p$ inequality (or the $RH_p$ inequality) implies that
the weight~$\om$ is doubling, and the dyadic $A_p$ inequality
implies that~$\om$ is dyadic doubling. However, the dyadic
$RH_p$ inequality does not imply that $\om$ is dyadic doubling,
which is why the dyadic doubling assumption is needed in the
definition of~$RH_p^d$.

We define the $\delta$-dyadic classes $A_p^\delta$ and
$RH_p^\delta$ similarly, using the collection~$\Dd$ of
intervals defined at the start of
Section~\ref{sec:dyadicintervals}.

It is shown in~\cite{BR} that a weight $\om$ belongs to $A_\infty$
if and only if $\om$ belongs to $RH_1$. Thus, $RH_1=A_\infty$ as
sets. Moreover, the constants are related by
\[
    \frac{1}{e} \, RH_1(\om)
    \leq A_\infty(\om)
    \leq C \, \frac{e^{e^{RH_1(\om)}}}{e^{RH_1(\om)}},
\]
where $C$ is independent of $RH_1(\om)$. The constant $1/e$ is
sharp, and the right-hand inequality is sharp in~$RH_1(\om)$.
The same proofs go through for the dyadic case.

We note that $A_\infty$ is the union of the $A_p$ classes, which are
nested and increasing as $p\rightarrow\infty$, and also $RH_1$ is
the union of the $RH_p$ classes, which are nested and decreasing as
$p\rightarrow\infty$. Thus,
$$
  A_\infty = \bigcup_{1\leq p< \infty} A_p = \bigcup_{1< p\leq \infty}
  RH_p =RH_1.
$$
See for example \cite{Gar}, \cite{Gra} or \cite{GCRF} for the theory
and history of $A_p$ weights, $1\leq p\leq \infty$, and $RH_p$
weights for $1<p<\infty$. The class $RH_\infty$ was defined in
\cite{CN}, and the class $RH_1$ was defined in \cite{BR} and, via an
equivalent definition, in \cite{HP}.

\begin{thm}\label{thm:intersectionoftranslatesApRHpdbl}
    Suppose $\delta\in\R$ is far from dyadic
    rationals, in the sense that condition (\ref{eqn:meicondR})
    holds. Then the following assertions hold:
    \begin{enumerate}
        \item[\textup{(a)}] $\om$ is a doubling weight if
            and only if $\om$ is dyadic doubling with
            respect to $\D$ and with respect to $\Dd$.

        \item[\textup{(b)}] $A_p = A_p^d \cap
            A_p^{\delta}$, for each $p$ with $1 \leq p \leq \infty$.

        \item[\textup{(c)}] $RH_p = RH_p^d \cap
            RH_p^{\delta}$, for each $p$ with $1 \leq p \leq \infty$.
    \end{enumerate}
    Bounds for the constants are given in the proof below. In
    particular, $A_p(\om)$ depends only on $A_p^d(\om)$ and
    $A_p^\delta(\om)$, and vice versa, and similarly for the other
    cases.
\end{thm}

It is immediate that if $\om$ is doubling then it is dyadic
doubling with respect to both $\D$ and $\Dd$, and similarly if
$\om$ lies in $A_p$ (respectively $RH_p$) then $\om$ lies in
both $A_p^d$ and $A_p^{\delta}$ (respectively $RH_p^d$ and
$RH_p^{\delta}$). In proving the other direction, the key point
is that the average of $\om$ over an interval $Q$ is comparable
to the average of $\om$ over the interval $I\supset Q$
guaranteed by Mei's proposition. More precisely, we have the
following proposition.

\begin{prop}\label{prop:comparable}
    If $\delta$ is far from the dyadic rationals  and
    $\om$ is both dyadic $\D$-doubling and dyadic
    $\Dd$-doubling with constant $C_{\textup{dy}}$, then given
    an interval~$Q$ and given the dyadic interval~$I$
    guaranteed by Proposition~\ref{prop:Mei}, we have
    \begin{equation}\label{eqn:comparable}
       (C_\textup{dy})^{-\log_2 (4C(\delta))} \intav_I \om
        \leq \intav_Q \om
        \leq C(\delta) \intav_I \om.
    \end{equation}
\end{prop}

\begin{proof} 
    Let $N$ be the unique integer such that $2^{N - 1} < C(\delta) \leq
    2^N$. Then $N + 1 < \log_2 (4C(\delta))$, and
    \[
        \frac{|I|}{2^{N+1}}
        \leq \frac{|I|}{2C(\delta)}
        \leq \frac{|Q|}{2}.
    \]
    Therefore, considering the $2^{N + 1}$ subintervals~$J$ of
    $I$ of length $|J| = |I|/2^{N + 1}$ ($J$ dyadic if
    $I\in\D$, $\delta$-dyadic if $I\in\Dd$), we can see that
    one of these intervals~$J$ must be completely contained in~$Q$. For
    this~$J$, we have
    \[
        \intav_I \om
        = \frac{1}{|I|} \int_I \om
        \leq (C_\textup{dy})^{N + 1} \frac{1}{|I|}  \int_J \om
        \leq  (C_\textup{dy})^{\log_2 (4C(\delta))} \frac{1}{|Q|}\int_Q \om
        = (C_\textup{dy})^{\log_2 (4C(\delta))} \intav_Q \om.
    \]

    Moreover, since $Q\subset I$, $|I| \leq C(\delta)|Q|$, and $\om
    \geq 0$, we have
    \[
        \intav_Q \om
        = \intq \om
        \leq \frac{C(\delta)}{|I|} \int_I \om
        = C(\delta) \intav_I \om.
        \qedhere
    \]
\end{proof}

We note that if $\om$ belongs to $A_p^d$ or $RH_p^d$ then $\om$ is
dyadic $\D$-doubling. Similarly $A_p^{\delta}$ or $RH_p^{\delta}$
functions are dyadic $\D^\delta$-doubling. Thus functions in
$A_p^d\cap A_p^{\delta}$ or in $RH_p^d\cap RH_p^{\delta}$ have the
comparability property~\eqref{eqn:comparable}, by
Proposition~\ref{prop:comparable}.

\begin{proof}[Proof of
    Theorem~\ref{thm:intersectionoftranslatesApRHpdbl}]
    We need only prove the reverse inclusions ($\supset$).

    (a) Suppose $\om$ is dyadic doubling with respect to
    both~$\D$ and $\Dd$. We show that $\om$ is doubling. Take
    an interval $Q$ in $\R$. The double $\widetilde{Q}$ of~$Q$
    is the interval $\widetilde{Q}$ that has the same midpoint
    as~$Q$ and twice the length: $|\widetilde{Q}| = 2|Q|$. Let
    $I$ be an interval of the type guaranteed by
    Proposition~\ref{prop:Mei} applied to~$\widetilde{Q}$. Take
    $N$ such that $2^{N-1} \leq C(\delta) < 2^N$. Then we have
    $|I| \leq 2^N|\widetilde{Q}|$.

    Consider the dyadic subintervals~$J$ of $I$ at scale $|J| =
    2^{-N-2}|I|$. These dyadic subintervals $J$ satisfy
    $|J| \leq |Q|/2$, which implies that $Q$ contains at
    least one such dyadic subinterval, denoted by $J$. Now we
    have
    \[
        \int_{\widetilde{Q}} \om
        \leq  \int_I \om
        \leq C_{\textup{dy}}^{N + 2} \int_J \om
        \leq C_{\textup{dy}}^{N + 2} \int_Q \om
        \leq C_{\textup{dy}}^{\log_2(2^3 C(\delta))} \int_Q \om.
    \]
    Thus $\om$ is doubling, with doubling constant at
    most~$C_{\textup{dy}}^{\log_2(2^3 C(\delta))}$.

    (b) Suppose $\om$ belongs to both $A_p^d$ and
    $A_p^{\delta}$, for some $p$ with $1 < p < \infty$. Take
    an interval~$Q$ in $\R$. Let $I$ be the interval guaranteed
    by Proposition~\ref{prop:Mei} applied to~$Q$. Then from
    Proposition~\ref{prop:comparable}, we have
    \[
        \intav_Q \om
        \leq C(\delta) \intav_I \om.
    \]
    Similarly, since $\om\in A_p^d\cap A_p^\delta$ implies that
    $\om^{-{1\over p-1}}$ belongs to $A_{p'}^d\cap
    A_{p'}^\delta$ where $p'$ is the conjugate index of $p$, we
    have
    \[
        \intav_Q \left(\frac{1}{\om}\right)^\frac{1}{p-1}
        \leq C(\delta) \intav_I \left(\frac{1}{\om}\right)^\frac{1}{p-1}.
    \]
    Therefore
    \[
        \left(\intav_Q \om\right)
            \left(\intav_Q \left(\frac{1}{\om}\right)^\frac{1}{p-1}\right)^{p-1}
        \leq C(\delta)^p \left(\intav_I \om\right)
            \left(\intav_I \left(\frac{1}{\om}\right)^\frac{1}{p-1}\right)^{p-1}
        \leq C(\delta)^p \max\{A_p^d(\om), A_p^{\delta}(\om)\}.
    \]
    Thus $\om$ lies in~$A_p$, with $A_p$ constant
    \[
        A_p(\om)
        \leq C(\delta)^p \max\{A_p^d(\om), A_p^{\delta}(\om)\}.
    \]

    For $p = 1$, suppose $\om$ belongs to both $A_1^d$
    and $A_1^{\delta}$, and let $V_1^d = \max\{A_1^d(\om),
    A_1^{\delta}(\om)\}$. Take~$Q$, $I$, and~$N$ as above. By
    Proposition~\ref{prop:comparable} we see that
    \[
        \intav_Q \om
        \leq C(\delta) \intav_I \om
        \leq C(\delta) \, V_1^d \, \essinf_{x\in I} \om(x)
        \leq C(\delta) \, V_1^d \, \essinf_{x\in Q} \om(x).
    \]
    Thus $\om\in A_1$ with
    \[
        A_1(\om) \leq C(\delta) \,
        \max\{A_1^d(\om), A_1^{\delta}(\om)\}.
    \]

    For $p = \infty$, if $\om\in A_\infty^d \cap A_\infty^{\delta}$, then there
    exist $p_1$ and~$p_2$ such that $\om\in A_{p_1}^d$ and
    $\om\in A_{p_2}^{\delta}$. Let $p = \max\{p_1, p_2\}$.
    Then $\om\in A_p^d \cap A_p^{\delta}$, so by the cases $1
    \leq p < \infty$ above, we have $\om\in A_p \subset
    A_\infty$.

    The dependence of the constants is shown in
    Figure~\ref{fig:depconstsAinfty}. In particular $A_\infty(\om)$ depends only on
    $\max\{A_\infty^d(\om), A_\infty^{\delta}(\om)\}$.

\setlength{\unitlength}{1cm}
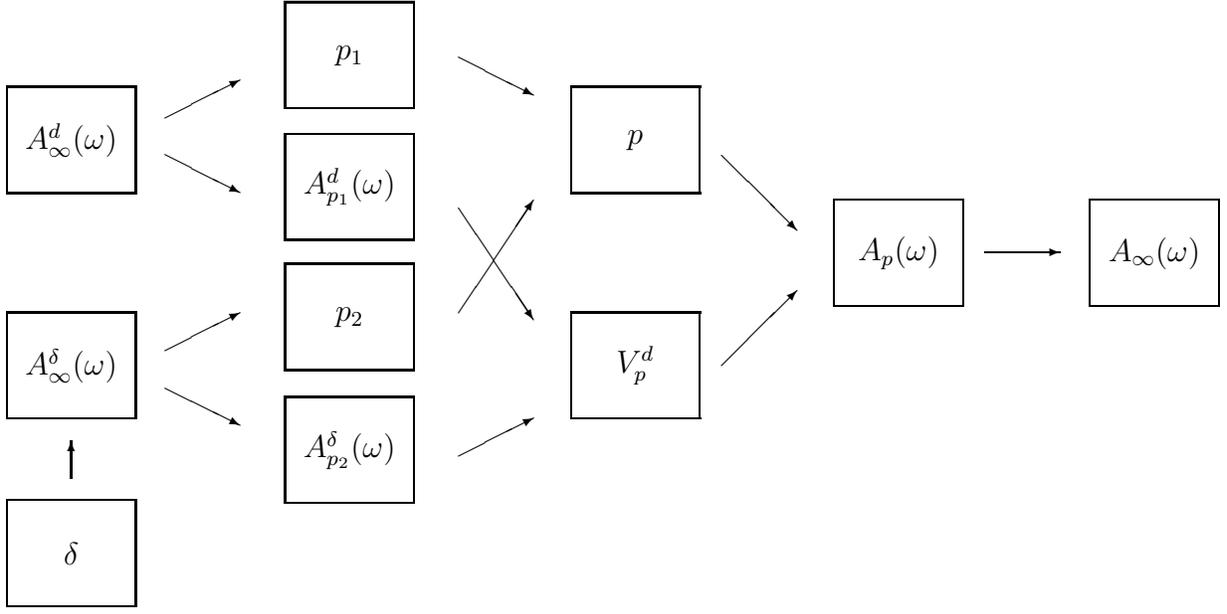
\begin{figure}[h]
    \begin{center}
        \begin{picture}(15,7)(0.5,-1)
            \put(0,3){\framebox(1.7,1.4){$A_\infty^d(\om)$}}
            \put(0,0){\framebox(1.7,1.4){$A_\infty^{\delta}(\om)$}}
            \put(0,-2.5){\framebox(1.7,1.4){$\delta$}}
            \put(3.7,4.125){\framebox(1.7,1.4){$p_1$}}
            \put(3.7,2.375){\framebox(1.7,1.4){$A_{p_1}^d(\om)$}}
            \put(3.7,0.645){\framebox(1.7,1.4){$p_2$}}
            \put(3.7,-1.125){\framebox(1.7,1.4){$A_{p_2}^{\delta}(\om)$}}
            \put(7.5,3){\framebox(1.7,1.4){$p$}}
            \put(7.5,0){\framebox(1.7,1.4){$V_p^d$}}
            \put(11.0,1.5){\framebox(1.7,1.4){$A_p(\om)$}}
            \put(14.4,1.5){\framebox(1.7,1.4){$A_\infty(\om)$}}
            \put(.85,-0.8){\vector(0,1){0.5}}
            \put(2.1,4.0){\vector(2,1){1.0}}
            \put(2.1,3.5){\vector(2,-1){1.0}}
            \put(2.1,0.9){\vector(2,1){1.0}}
            \put(2.1,0.4){\vector(2,-1){1.0}}
            \put(6.0,4.8){\vector(2,-1){1.0}}
            \put(6.0,1.4){\vector(2,3){1.0}}
            \put(6.0,2.8){\vector(2,-3){1.0}}
            \put(6.0,-0.5){\vector(2,1){1.0}}
            \put(9.5,3.5){\vector(1,-1){1.0}}
            \put(9.5,0.7){\vector(1,1){1.0}}
            \put(13.0,2.2){\vector(1,0){1.0}}
        \end{picture}
    \end{center}
    \vskip2cm
    \caption{Dependence of the constants in the proof of
        Theorem~1, for the case of $A_\infty$. Here $V_p^d
        := \max\{A_p^d(\om),A_p^{\delta}(\om)\}$.}
    \label{fig:depconstsAinfty}
\end{figure}


    (c) Suppose $\om$ belongs to $RH_p^d \cap RH_p^{\delta}$,
    for some $p$ with $1 < p < \infty$. Then $\om$ is both
    dyadic $\D$-doubling and dyadic $\D^\delta$-doubling with
    doubling constant $C_{\textup{dy}}$. Moreover, $\om^p$ is
    also both dyadic $\D$-doubling and dyadic
    $\D^\delta$-doubling, since for each $I\in \D$ or
    $\D^\delta$,
\begin{eqnarray*}
       \left(\intav_{\widetilde{I}} \om^p\right)^{1/p}
       &\leq& RH_p(\om)\intav_{\widetilde{I}} \om
            = \frac{RH_p(\om)}{2}\frac{1}{|I|}\int_{\widetilde{I}} \om
            \leq  \frac{RH_p(\om)}{2} C_{\textup{dy}} \frac{1}{|I|} \int_{I} \om\\
       &\leq& \frac{RH_p(\om)}{2} C_{\textup{dy}} \left(\intav_{I} \om^p\right)^{1/p},
\end{eqnarray*}
    which yields that
    \[
        \int_{\widetilde{I}} \om^p
        \leq 2\left(\frac{RH_p(\om)}{2}C_{\textup{dy}}\right)^p \int_{I} \om^p,
    \]
    where $\widetilde{I}$ is the dyadic parent of $I$. That is, $\om^p$
    is also both dyadic $\D$-doubling and dyadic
    $\D^\delta$-doubling with constant
    $2^{1-p}\left(RH_p(\om)C_{\textup{dy}}\right)^p$.

    Now take~$Q$, $I$, and~$N$ as in part (b). Since $\om$ and
    $\om^p$ are both dyadic $\D$-doubling and dyadic
    $\D^\delta$-doubling, Proposition~\ref{prop:comparable} implies that
    \[
        \left(\intav_Q \om^p\right)^{1/p}
        \leq C(\delta)^{1/p} \left(\intav_I \om^p\right)^{1/p}
        \qquad \text{and} \qquad
        \intav_I \om
        \leq (C_\textup{dy})^{\log_2 (4C(\delta))} \intav_Q \om.
    \]
    Thus
    \[
        \left(\intav_Q \om^p\right)^{1/p} \left(\intav_Q \om\right)^{-1}
        \leq C(\delta)^{1/p} (C_\textup{dy})^{\log_2 (4C(\delta))}\left(\intav_I \om^p\right)^{1/p}
             \left(\intav_I \om\right)^{-1}.
    \]
    So $\om$ belongs to~$RH_p$, and
    \[
        RH_p(\om)
        \leq C(\delta)^{1/p} \, (C_\textup{dy})^{\log_2 (4C(\delta))} \,
            \max\{RH_p^d(\om), RH_p^{\delta}(\om)\}.
    \]

    For $p = 1$, suppose $\om$ belongs to both $RH_1^d$ and
    $RH_1^{\delta}$. By the comment after
    Definition~\ref{def:RHp}, $\om$ belongs to both
    $A_\infty^d$ and $A_\infty^{\delta}$, with constants
    depending only on $RH_1^d(\om)$ and $RH_1^{\delta}(\om)$.
    So by part~(b) above, $\om$ belongs to $A_\infty$, and by
    the same comment, $\om$ belongs to $RH_1$, with constant
    $RH_1(\om)$ depending only on $RH_1^d(\om)$
    and~$RH_1^{\delta}(\om)$.

    For $p = \infty$, suppose $\om$ belongs to both $RH_\infty^d$ and
    $RH_\infty^{\delta}$, and let $V_\infty^d =
    \max\{RH_\infty^d(\om), RH_\infty^{\delta}(\om)\}$.
    Take~$Q$, $I$, and~$N$ as above. Then
    \[
        \esssup_{x\in Q} \om(x)
        \leq \esssup_{x\in I} \om(x)
        \leq V_\infty^d \intav_I \om
        \leq V_\infty^d \, (C_\textup{dy})^{\log_2 (4C(\delta))} \intav_Q \om.
    \]
    Thus $\om$ belongs to $RH_\infty$ and
    \[
        RH_\infty(\om) \leq (C_\textup{dy})^{\log_2 (4C(\delta))} \,
            \max\{RH_\infty^d(\om), RH_\infty^{\delta}(\om)\}.
        \qedhere
    \]
\end{proof}


\subsection{Multiparameter results for $A_p$ and $RH_p$}
\label{subsec:productresultsApRHpdbl}

We extend the above results for $A_p(\R)$ and $RH_p(\R)$ to the
multiparameter ($\R\otimes\cdots\otimes\R)$ setting. For ease
of notation, we write the statements and proofs for the product
space $\R\otimes\R$ of two factors. The same proofs go through
for arbitrarily many factors.

As noted in~\cite{PWX}, the theory of product weights was
developed by K.-C.~Lin in his thesis~\cite{L}, while the dyadic
theory was developed in Buckley's paper~\cite{B}. The product
$A_p$ and $RH_p$ weights and the product doubling weights, and
their dyadic analogues, are defined exactly as in
Definitions~\ref{def:Ap}--\ref{def:RHp}, with intervals in~$\R$
being replaced by rectangles in~$\R\otimes\R$. It follows that
a product weight belongs to $A_p(\R\otimes\R)$ if and only if
it belongs to $A_p(\R)$ in each variable separately.

To be precise, $\om\in A_p(\R\otimes\R)$ if and only if
$\om(\cdot,y) \in A_p(\R)$ uniformly for a.e.~$y\in\R$ and
$\om(x,\cdot) \in A_p(\R)$ uniformly for a.e.~$x\in\R$. In one
direction this is a consequence of the Lebesgue Differentiation
Theorem, letting one side of the rectangle shrink to a point.
The converse uses the equivalence between $\om\in
A_p(\R\otimes\R)$ and maximal inequality of the strong maximal
function~\cite[p.83]{S}. Further, the $A_p(\R\otimes\R)$
constant depends only on the two $A_p(\R)$ constants, and vice
versa.

The analogous characterizations in terms of the separate
variables hold for product~$RH_p$ weights and for product
doubling weights, and for the dyadic product $A_p$, $RH_p$, and
doubling weights.

We denote by $A_p^{d,d} = A_p^{d,d}(\R\otimes\R)$ the class of
strong dyadic weights, meaning the weights $\om(x,y)$ such that
\begin{enumerate}
    \item[(i)] for a.e.~fixed $y$, $\om(\cdot,y)$ lies in
        $A_p^d(\R)$, and

    \item[(ii)] for a.e.~fixed $x$, $\om(x,\cdot)$ lies in
        $A_p^d(\R)$,
\end{enumerate}
with uniform $A_p^d(\R)$ constants. The class $A_p^{d,\delta}$
is the same as $A_p^{d,d}$ except that $\om(x,\cdot)$ lies in
$A_p^\delta(\R)$, using in the second variable the translated
dyadic grid~$\Dd$. Similarly for $A_p^{\delta,d}$ and
$A_p^{\delta,\delta}$, and for the corresponding variations of
$RH_p^{d,d}$.

\begin{thm}\label{thm:product intersectionoftranslatesApRHpdbl}
    Suppose $\delta\in\R$ is far from dyadic
    rationals, in the sense that (\ref{eqn:meicondR})
    holds. Then the following assertions hold:
    \begin{enumerate}
        \item[\textup{(a)}] A weight $\om(x,y)$ is a
            product doubling weight if and only if $\om$ is
            dyadic doubling with respect to each of
            $\D\times\D$, $\D\times\Dd$, $\Dd\times\D$, and
            $\Dd\times\Dd$.

        \item[\textup{(b)}] For $1 \leq p \leq \infty$,
            biparameter $A_p$ is the intersection of four
            translates of biparameter dyadic $A_p$:
            \[
                A_p(\R\otimes\R)
                = A_p^{d,d}(\R\otimes\R)
                    \cap A_p^{d,\delta}(\R\otimes\R)
                    \cap A_p^{\delta,d}(\R\otimes\R)
                    \cap A_p^{\delta,\delta}(\R\otimes\R).
            \]

        \item[\textup{(c)}] For $1 \leq p \leq \infty$,
            biparameter $RH_p$ is the intersection of four
            translates of biparameter dyadic $RH_p$:
            \[
                RH_p(\R\otimes\R)
                = RH_p^{d,d}(\R\otimes\R)
                    \cap RH_p^{d,\delta}(\R\otimes\R)
                    \cap RH_p^{\delta,d}(\R\otimes\R)
                    \cap RH_p^{\delta,\delta}(\R\otimes\R).
            \]
    \end{enumerate}
    The constant $A_p(\om)$ depends only on $A_p^{d,d}(\om)$,
    $A_p^{d,\delta}(\om)$, $A_p^{\delta,d}(\om)$, and
    $A_p^{\delta,\delta}(\om)$, and vice versa, and similarly
    for the other classes.

    In the case of $k$~parameters, the analogous results hold
    using the intersection of $2^k$ translates of the dyadic
    classes.
\end{thm}

\begin{proof}
The proof is by iteration of the one-parameter argument. We
sketch the case of $A_p$ for $1 < p < \infty$. The other cases
are similar. Take $\om\in A_p(\R\otimes\R)$. By our
one-parameter result
(Theorem~\ref{thm:intersectionoftranslatesApRHpdbl}), for
almost every~$y$ we have $\om(\cdot,y)\in A_p(\R) = A_p^d(\R)
\cap A_p^\delta(\R)$, and similarly in the second variable.
Thus $\om\in A_p^{d,d}(\R\otimes\R)$, and similarly $\om\in
A_p^{d,\delta}(\R\otimes\R)$, $\om\in
A_p^{\delta,d}(\R\otimes\R)$, and $\om\in
A_p^{\delta,\delta}(\R\otimes\R)$. Conversely, if $\om$ is in
the intersection of the four dyadic spaces, then for a.e.~$y$,
$\om(\cdot,y)$ belongs to $A_p^d(\R) \cap A_p^\delta(\R) =
A_p(\R)$. Similarly, for a.e.~$x$, $\om(x,\cdot)$ belongs to
$A_p(\R)$. Therefore $\om\in A_p(\R\otimes\R)$. Moreover, the
claimed dependence of the constants follows immediately from
the one-parameter result.
\end{proof}

\section{$\bmo$ and product $\bmo$}
\label{sec:resultsBMO}
\setcounter{equation}{0}

We extend Mei's $\bmo$ result to the biparameter case. We begin
by recalling some observations and background results; for
details see~\cite{CF, S}. Next we give a proof of Mei's
one-parameter result, still using Mei's key lemma but
expressing \bmo\ in terms of Carleson measures. Then we extend
this proof to the multiparameter case.

\begin{prop}\label{prop:psiprops}
    Let $\psi \in C_\text{c}^\infty(\R)$ be a smooth function,
    supported in the interval $[-1,1]$, such that $\int \psi(t)
    \, dt = 0$. For $y > 0$ let $\psi_y(t) := \frac{1}{y} \,
    \psi\big(\frac{t}{y}\big)$. For $t\in\R$ and $y > 0$ let
    $I_{t,y} := [t - y, t + y]$. Then
    \begin{enumerate}
        \item[\textup{(i)}] if $(t,y)\in T(I_0)$ then
            $I_{t,y}\subset 3I_0$, where $3I_0$ is the
            interval with the same midpoint as $I_0$ and
            length $|3I_0| = 3|I_0|$,

        \item[\textup{(ii)}] $\supp \psi_y \subset [-y,y]$,

        \item[\textup{(iii)}] $\supp \psi_y(t - \cdot)
            \subset I_{t,y}$,

        \item[\textup{(iv)}] for $I\in\D$ or
            $I\in\D^\delta$, and for the Haar function
            $h_I$ on~$I$, if $I\cap I_{t,y} = \emptyset$
            then $h_I * \psi_y(t) = 0$, and

        \item[\textup{(v)}] if $I_{t,y}\subset Q_l$ or
            $I_{t,y}\subset Q_r$, where $Q_l$ and $Q_r$ are
            the left and right halves respectively of an
            interval~$Q$, then $h_I * \psi_y(t) = 0$.
    \end{enumerate}
\end{prop}
We omit the (elementary) proofs, except to note that part~(v)
holds since $h_Q$ is constant on each of $Q_l$ and $Q_r$,
$\psi_y$ is supported in $I_{t,y}$, and $\int_{\R} \psi = 0$.
Now we impose an additional condition (the
Calder\'on--Torchinsky condition) on $\psi$ as follows: there
exists a constant $C_{\psi}$ such that for any $\xi\not= 0$,
\begin{eqnarray}\label{eqn:CT condition}
   \int_0^\infty{|\widehat{\psi}(\xi t)|^2\over t}dt\leq C_{\psi}.
\end{eqnarray}

\begin{defn}
    For $f\in L^1_{\textup{loc}}(\R)$ and for each dyadic
    interval~$J$, define the projection $P_J$ of $f$ by
    \[
        P_J f(x)
        := \sum_{I\in\D, I \subset J} (f,h_I) h_I(x).
    \]
\end{defn}


If $f\in L^2(\R)$, then $\Vert f \Vert_{L^2(\R)}^2 =
\sum_{I\in\D} (f,h_I)^2$. If $f\in L^2(\R)$, then the following
standard Littlewood--Paley $L^2$ estimate holds:
\[
    \int\!\!\!\int_{(t,y)\in\R\otimes\R_+} \,
        |f * \psi_y(t)|^2 \, \frac{dt \, dy}{y}\leq C_{\psi}\|f\|_{L^2(\R)}^2
\]
$\psi\in C_\text{c}^\infty(\R)$, $\int \psi=0$ and satisfies (\ref{eqn:CT condition}).

\begin{defn}\label{def:AverageBMOd}
    A locally integrable function~$f$ belongs to the
    \emph{dyadic BMO space} $\bmo_d(\R)$ if  there
    is a constant~$C$ such that
    \begin{equation}\label{eqn:AverageBMOd}
    \|f\|_{\bmo_d(\R)} := \sup_{I\in\D} \intav_I |f(x) - f_I| \, dx < \infty,
    \end{equation}
    where $f_I := \intav_I f$.
\end{defn}
It follows from the John--Nirenberg theorem \cite[Theorem 2.1, p.230]{Gar} that for each $p>1$, the expression
$$
   \|f\|_{\bmo_{d,p}(\R)} := \sup_{I\in\D} \Big(\intav_I |f(x) - f_I|^p \, dx\Big)^{1/p}
$$
is comparable to $\|f\|_{\bmo_d(\R)}$.

We also have the following equivalent definition of $\bmo_d(\R)$ in
terms of dyadic Carleson measures.
\begin{defn}\label{def:CarlesonBMOd}
    A locally integrable function~$f$ belongs to the
    \emph{dyadic BMO space} $\bmo_d(\R)$ if there
    is a constant~$C$ such that for all dyadic intervals~$J$,
    \begin{equation}\label{eqn:CarlesonBMOd}
        \sum_{I\in\D, I\subset J} (f,h_I)^2 \leq C|J|.
    \end{equation}
\end{defn}
We note that if in Definition \ref{def:CarlesonBMOd} we allow
$J$ to range over all intervals in $\R$, not only dyadic
intervals in $\R$, we recover the same dyadic \bmo\ space
$\bmo_d(\R)$, with comparable norms. This observation follows
from Proposition \ref{prop:dyadiccover} below and the fact that
the sum in inequality (\ref{eqn:CarlesonBMOd}) is over only
dyadic intervals $I$.

The equivalence of conditions (\ref{eqn:AverageBMOd}) and
(\ref{eqn:CarlesonBMOd}) can be seen as follows. First, suppose $f$
satisfies (\ref{eqn:AverageBMOd}). Then for each dyadic interval $J$,
\begin{eqnarray}
   {1\over |J|} \sum_{I\in\D, I\subset J} (f,h_I)^2
   &=& {1\over |J|} \sum_{I\in\D, I\subset J} \big((f-f_J)\chi_J,h_I\big)^2\nonumber\\
   &\leq& {1\over |J|} \sum_{I\in\D} \big((f-f_J)\chi_J,h_I\big)^2\nonumber\\
   &\leq& {1\over |J|} \|(f-f_J)\chi_J\|_{L^2(\R)}^2  \label{eqn:average to Carleson}\\
   &=& {1\over |J|} \int_J|f(x)-f_J|^2\,dx\nonumber\\
   &\leq& C \|f\|^2_{\bmo_d(\R)} ,\nonumber
\end{eqnarray}
which shows that $f$ satisfies condition
(\ref{eqn:CarlesonBMOd}). Here we use $\chi_J$ to denote the
characteristic function on $J$; in the first inequality we use
the fact that $(C,h_I)=0$ for any constant $C$; and in the last
inequality, we use the fact that the dyadic \bmo\ norms
$\|\cdot\|_{\bmo_d(\R)}$ and $\|\cdot\|_{\bmo_{d,2}(\R)}$ are
equivalent.

Conversely, suppose $f$ satisfies (\ref{eqn:CarlesonBMOd}). Then for
each dyadic interval $J$,
\begin{eqnarray}
   {1\over |J|} \int_J|f(x)-f_J|\,dx
   &\leq& \Big({1\over |J|} \int_J|f(x)-f_J|^2\,dx\Big)^{1/2}\nonumber\\
   &\leq& \Big({1\over |J|} \int |f(x)-f_J|^2\chi_J(x)\,dx\Big)^{1/2}\nonumber\\
   &\leq& \Big({1\over |J|} \sum_{I\in\D}\big((f-f_J)\chi_J, h_I\big)^2\Big)^{1/2}\label{eqn:Carleson to average}\\
   &\leq& \Big({1\over |J|} \sum_{I\in\D, \ I\subset J} \big( (f-f_I), h_I\big)^2\Big)^{1/2}\nonumber\\
   &\leq& C^{1/2},\nonumber
\end{eqnarray}
where the constant $C$ in the last inequality is from condition
(\ref{eqn:CarlesonBMOd}).

From the above estimates, we see that the smallest constant $C$
in condition (\ref{eqn:CarlesonBMOd}) is comparable to
    $\Vert f \Vert_{\bmo_d(\R)}^2$.

We define $\bmo_\delta(\R)$ similarly, in terms of both
averages and Carleson conditions, with respect to
the~collection $\Dd$ of translated dyadic intervals. Here
$\delta\in\R$, and $\Dd$ is defined at the start of
Section~\ref{sec:dyadicintervals}.

\begin{prop}\label{prop:dyadiccover}
    Let $K$ be any interval in~$\R$, dyadic or not. Then $K$ is
    contained in the union of two adjacent dyadic
    intervals $J_1$ and $J_2$ of equal length, i.e., $K\subset J_1
    \cup J_2$, with
    \[
        \frac{|J_1|}{2}
        = \frac{|J_2|}{2}
        < |K|
        \leq |J_1|
        = |J_2|.
    \]
\end{prop}
\begin{proof}
Let $N$ be the unique integer such that $2^{N-1} < |K| \leq
2^N$. Let $J_1$ be the dyadic interval of length $2^N$ that
contains the left (or right) endpoint of~$K$. If $K\subset J_1$, we are
done. If $K\not\subset J_1$, then the right (or left) endpoint of~$K$
lies in the dyadic interval $J_2$ of length~$2^N$ immediately
to the right of $J_1$.
\end{proof}

\begin{thm}\label{thm:intersectionoftranslatesBMOR}
    Suppose $\delta\in\R$ is far from dyadic
    rationals, in the sense of condition~\eqref{eqn:meicondR}.
    Then $\bmo = \bmo_d \cap \bmo_{\delta}$. Moreover, we
    have
    $$
    \max\{ \|f\|_{\bmo_d(\R)},\|f\|_{\bmo_\delta(\R)} \}
    \leq \|f\|_{\bmo(\R)}
    \leq (C\cdot C(\delta))^{1/2}
        \max\{ \|f\|_{\bmo_d(\R)},\|f\|_{\bmo_\delta(\R)} \},
    $$
    where $C$ depends only on $C_\psi$ in (\ref{eqn:CT condition}).
\end{thm}

\begin{proof}
The inclusion $\bmo \subset \bmo_d \cap \bmo_{\delta}$ is an
immediate consequence of the definition
\begin{eqnarray}\label{def:BMO average}
    \bmo := \{f\in L^1_{\textup{loc}} :
    \|f\|_* := \sup_Q \intq |f(x) - f_Q| \, dx < \infty\}
\end{eqnarray}
in terms of averages $f_Q := (1/|Q|)\int_Q f$, since for
$\bmo_d$ and $\bmo_{\delta}$ we are taking the supremum over
fewer intervals than for~$\bmo$. Moreover, we have $\max\{
\|f\|_{\bmo_d(\R)},\|f\|_{\bmo_\delta(\R)} \}\leq
\|f\|_{\bmo(\R)}$.

Now we prove the other inclusion. Our proof, which replies on
the Carleson-measure characterization of \bmo, is more
complicated than the original proof in \cite{Mei}. We choose to
give this proof because it readily generalizes to the
multiparameter case (see Theorem
\ref{thm:intersectionoftranslatesBMORn}). Suppose $f$ belongs
to both $\bmo_d(\R)$ and~$\bmo_{\delta}(\R)$. Choose a $\psi$
as in Proposition \ref{prop:psiprops} and satisfying
(\ref{eqn:CT condition}). We must show that there is a positive
constant $C$ independent of $I_0$ such that the inequality
\[
    \int\!\!\!\int_{T(I_0)} \, |f * \psi_y(t)|^2 \,
        \frac{dt \, dy}{y}
    \leq C|I_0|
\]
holds for all intervals~$I_0$, where the constant $C$ is
comparable to $\|f\|_{\bmo(\R)}^2$.

Fix an interval $I_0\subset\R$. For each point~$(t,y)$
in~$T(I_0)$, let $I_{t,y} := (t - y, t + y)$ be the interval of
length $2y$ centered at~$t$. By Proposition~\ref{prop:Mei}, for
each point $(t,y) \in T(I_0)$ we may choose an interval
$I_{t,y}^*$ such that $I_{t,y} \subset I_{t,y}^*$, $|I_{t,y}^*|
\leq C(\delta) |I_{t,y}|$, and either $I_{t,y}^*\in\D$ or
$I_{t,y}^*\in\Dd$. Let
\[
    \F_1
    := \{(t,y) \in T(I_0) \mid I_{t,y}^*\in\D\},
    \qquad
    \F_2
    := \{(t,y) \in T(I_0) \mid I_{t,y}^*\in\Dd\}.
\]
So $T(I_0) = \F_1 \cup \F_2$, and $\F_1 \cap \F_2 = \emptyset$.
Now we have
\[
    \int\!\!\!\int_{(t,y) \in T(I_0)} \,
        |f * \psi_y(t)|^2 \, \frac{dt \, dy}{y}
    = \underbrace{\int\!\!\!\int_{(t,y) \in \F_1} \,
        |f * \psi_y(t)|^2 \, \frac{dt \, dy}{y}}_{(G_1)}
    + \underbrace{\int\!\!\!\int_{(t,y) \in \F_2} \,
        |f * \psi_y(t)|^2 \, \frac{dt \, dy}{y}}_{(G_2)}.
\]
%
%
%
%
It suffices to control the term $(G_1)$ since the estimate for
the term $(G_2)$ is similar. Replacing $f$ by its Haar
expansion, we see that
\begin{equation}\label{eqn:fstarpsi}
    f * \psi_y(t)
    = \sum_{I\in\D} (f,h_I) h_I * \psi_y(t)
    = \sum_{I\in\D, I\cap I_{t,y} \neq \emptyset}
        (f,h_I) h_I * \psi_y(t),
\end{equation}
since by Proposition~\ref{prop:psiprops}(iv), $h_I
* \psi_y(t)$ can only be nonzero if $I\cap I_{t,y}
\neq\emptyset$.

For each $(t,y)\in\mathcal{F}_1$, we have $I_{t,y}\subset
3I_0$, by Proposition~\ref{prop:psiprops}(i).

Fix $(t,y)\in\F_1$. We split the sum in
equation~\eqref{eqn:fstarpsi} at the scale of $2^N|3I_0|$,
where $N > 0$ is a constant to be determined later but
independent of $f$, $t$, $y$ and~$I_0$. Let $k_0$ be the unique
integer such that
$$
   2^{-k_0} \leq |3I_0| < 2^{-k_0 + 1}.
$$
Now,
\begin{eqnarray*}
    f * \psi_y(t)
    & = & \underbrace{\sum_{k = k_0-N-1}^\infty
            \sum_{I\in\D_k, I\cap I_{t,y} \neq \emptyset}
            (f,h_I) h_I * \psi_y(t)}_{g_{11}}
        + \underbrace{\sum_{k = -\infty}^{k_0 - N - 2}
            \sum_{I\in\D_k, I\cap I_{t,y} \neq \emptyset}
            (f,h_I) h_I * \psi_y(t)}_{g_{12}}.
\end{eqnarray*}

\noindent\textbf{For the sum  $g_{12}$:} we first show that
each term in the sum $g_{12}$ over large intervals is zero, if
$N$ is chosen appropriately. Let $N$ be the unique integer such
that
\begin{equation}\label{eqn:NCdelta}
    2^N \leq 2C(\delta) < 2^{N + 1}.
\end{equation}
(Note that $N \geq 2$, since $d(\delta) < 1$ and so $2C(\delta)
= 4/d(\delta) > 2^2$.) We will use the right-hand inequality
for our estimate of $g_{12}$, and the left-hand inequality for
$g_{11}$.

If the interval $I$ appears in the sum $g_{12}$, we have
\[
    |I|
    \geq 2^{-k_0 + N + 2}
    > 2^{N + 1} |3I_0|
    > 2C(\delta) |I_{t,y}|
    \geq 2|I_{t,y}^*|.
\]
Since the intervals $I$ and $I_{t,y}^*$ both belong to the same
dyadic grid~$\D$ and $|I| > |I_{t,y}^*|$, it follows that
either $I$ and $I_{t,y}^*$ are disjoint or $I\supsetneqq
I_{t,y}^*$. If the former, then $h_I * \psi_y(t) = 0$. If the
latter, then since $I_{t,y}\subset I_{t,y}^*\subsetneqq I$ we
see that $I_{t,y}$ is contained in either the left half of~$I$
or the right half of $I$, and so by
Proposition~\ref{prop:psiprops}(v), $h_I
* \psi_y(t) = 0$. Thus the sum $g_{12}$ is zero.

\noindent\textbf{For the sum  $g_{11}$:}
For each interval $I$ that appears in~$g_{11}$, we have $|I|
\leq 2^{-k_0 + N + 1} \leq 2^{N+1}|3I_0|$ and $I\cap 3I_0 \neq
\emptyset$. It follows that each such interval~$I$ is contained
in the interval $2^{N + 1}9I_0$ that has the same midpoint as
$I_0$ and length $2^{N + 1}|9I_0|$. For brevity, let
\[
    J_0 := 2^{N + 1}9I_0.
\]
We reiterate that if $I\cap I_{t,y} \neq \emptyset$ then
$I\subset J_0$.

Then
\begin{eqnarray*}
    g_{11}
    & := & \sum_{k = k_0-N-1}^\infty
        \sum_{I\in\D_k, I\cap I_{t,y} \neq \emptyset}
        (f,h_I) h_I * \psi_y(t)\\
    & = & \sum_{k = k_0 - N - 1}^\infty \sum_{I\in\D_k, I\subset J_0}
        (f,h_I) h_I * \psi_y(t)\\
    & = & \sum_{I\in\D, I\subset J_0} (f,h_I) h_I * \psi_y(t).
\end{eqnarray*}
The third equality holds because if $I\subset J_0$, $I\in\D_k$
and $k<k_0-N-1$, then $h_I*\psi_y(t)=0$ by Proposition
\ref{prop:psiprops} and the argument for $g_{12}$ above.

\m
As a consequence, and applying the Carleson
condition~\eqref{eqn:CarlesonBMOd} for $f\in\bmo_d(\R)$ and the
interval~$J_0$, we see that
\begin{align*}
    (G_1)
    &= \int\!\!\!\int_{\F_1} \, |g_{11}|^2 \,
        \frac{dt \, dy}{y}
    \leq \int\!\!\!\int_{\F_1} \, \bigg|\sum_{I\in\D, I\subset J_0}
        (f,h_I) h_I * \psi_y(t)\bigg|^2 \, \frac{dt \, dy}{y} \\
    &= \int\!\!\!\int_{\F_1} \, |
        P_{J_0} f * \psi_y(t)|^2 \, \frac{dt \, dy}{y}
    \leq \int\!\!\!\int_{(t,y)\in\R\otimes\R_+} \,
        |P_{J_0} f * \psi_y(t)|^2 \, \frac{dt \, dy}{y} \\
    &\leq C\Vert P_{J_0} f\Vert_{L^2(\R)}^2 \\
    &= C\sum_{I\in\D} (P_{J_0} f,h_I)^2 \\
    &= C\sum_{I\in\D, I\subset J_0} (f,h_I)^2 \\
    &\leq C |J_0| \, \Vert f\Vert_{\bmo_d(\R)}^2 \\
    &\leq C 2^{N+1}|I_0| \, \Vert f\Vert_{\bmo_d(\R)}^2 \\
    &\leq C\cdot C(\delta) |I_0| \, \Vert f\Vert_{\bmo_d(\R)}^2,
\end{align*}
where $C$ depends only on $C_\psi$ in (\ref{eqn:CT condition}).

It follows that
\[
   (G_1)
   \leq C\cdot C(\delta) |I_0| \, \Vert f\Vert_{\bmo_d(\R)}^2.
\]

In the same way, we obtain the analogous estimate for $(G_2)$, with
$\Vert f\Vert_{BMO_d(\R)}$ replaced by $\Vert
f\Vert_{BMO_{\delta}(\R)}$:
\[
   (G_2)
   \leq C\cdot C(\delta) |I_0| \, \Vert f\Vert_{\bmo_{\delta}(\R)}^2.
\]
Therefore
\begin{equation}\label{eqn:BMOlastinequality}
    \Vert f\Vert_{\bmo(\R)}
    \leq (C\cdot C(\delta))^{1/2}
        \max\{\Vert f\Vert_{\bmo_{d}(\R)},\Vert f\Vert_{\bmo_{\delta}(\R)}\},
\end{equation}
as required.
\end{proof}



\bigskip\bigskip

We now turn to the product case. For simplicity we discuss the case of two parameters.

A locally integrable function~$f$ on $\R\otimes\R$ belongs to the product BMO space
$\bmo(\R\otimes\R)$ if there exists a positive constant
$C$ such that for every open set $\Omega\subset\R\otimes\R$ with
finite measure, the following inequality holds:
\begin{eqnarray}\label{eqn:product BMO}
    \int\!\!\!\int_{T(\Omega)} \, |f * \psi_{y_1}\psi_{y_2}(t_1,t_2)|^2 \,
        \frac{dt_1 \, dy_1\, dt_2 \, dy_2}{y_1y_2}
    \leq C|\Omega|.
\end{eqnarray}
Here $T(\Omega) := \left\{(t_1,y_1,t_2,y_2) \mid
I_{t_1,y_1}\times I_{t_2,y_2}\subset\Omega \right\}$ is the
Carleson tent on~$\Omega$. Also
$\psi_{y_1}\psi_{y_2}(t_1,t_2)=y_1^{-1}y_2^{-1} \psi(t_1/y_1)
\psi(t_2/y_2)$, where $\psi$ is a function of the kind
described above in the one-parameter case.  The smallest such
$C$ is comparable to  $\Vert f \Vert_{\bmo(\R\otimes\R)}^2$.


Next we mention the following four types of dyadic product BMO
spaces $\bmo_{d,d}$, $\bmo_{d,\delta}$, $\bmo_{\delta,d}$ and
$\bmo_{\delta,\delta}$. They differ only in which of the dyadic
grids $\D$ and $\Dd$ is used in each variable. First, a locally
integrable function~$f$ on $\R\otimes\R$ belongs to
$\bmo_{d,d}(\R\otimes\R)$ if and only if there exists a
positive constant $C$ such that for each open set
$\Omega\subset\R\otimes\R$ with finite measure, the following
inequality
\begin{eqnarray}\label{eqn:product BMOdd}
    \sum_{R=I\times J\in \D\times\D, \ R\subset\Omega}(f,h_R)^2
    \leq C|\Omega|
\end{eqnarray}
holds, where $h_R = h_I\times h_J$, and $h_I$ and $h_J$ are the Haar
functions on the intervals $I\in\D$ and $J\in\D$, respectively.

Next, a locally integrable function~$f$
on $\R\otimes\R$ belongs to $\bmo_{d,\delta}(\R\otimes\R)$ if and
only if there exists a positive constant $C$ such that for each open
set $\Omega\subset\R\otimes\R$ with finite measure, the following
inequality
\begin{eqnarray}\label{eqn:product BMOd delta}
    \sum_{R=I\times J\in \D\times\D^\delta, \ R\subset\Omega}(f,h_R)^2
    \leq C|\Omega|
\end{eqnarray}
holds, where $h_R=h_I\times h_J$, $h_I$ and $h_J$ are the Haar
functions on the intervals $I\in\D$ and $J\in\Dd$, respectively.

We define $\bmo_{\delta,d}(\R\otimes\R)$ and
$\bmo_{\delta,\delta}(\R\otimes\R)$ similarly.

For simplicity we state and prove our multiparameter result for
two parameters. However the statement and proof go through for
arbitrarily many parameters.

\begin{thm}\label{thm:intersectionoftranslatesBMORn}
    Suppose $\delta\in\R$ is far from dyadic rationals: $\delta$
    satisfies condition~\eqref{eqn:meicondR}. Then
    \[
        \bmo(\R\otimes\R)
        = \bmo_{d,d}(\R\otimes\R) \cap
            \bmo_{d,\delta}(\R\otimes\R) \cap
            \bmo_{\delta,d}(\R\otimes\R) \cap
            \bmo_{\delta,\delta}(\R\otimes\R).
    \]
    Bounds for the constants are given in the proof below.

    In the case of $k$~parameters, the analogous result holds
    using the intersection of $2^k$ translates of the dyadic
    classes.
\end{thm}

\begin{proof}
We first note that $\bmo(\R\otimes\R) \subset
\bmo_{d,d}(\R\otimes\R) \cap \bmo_{d,\delta}(\R\otimes\R)\cap
\bmo_{\delta,d}(\R\otimes\R)\cap
\bmo_{\delta,\delta}(\R\otimes\R)$. This inclusion is not
trivial in the multiparameter setting. A proof (for biparameter
$\bmo$) was given in the Ph.D. thesis \cite{P} of J.~Pipher,
but the best proof of this result is in S.~Treil's paper
\cite{T}. There he shows that  $H^1(\R\otimes\R) \supset
H^1_{d,d}(\R\otimes\R)$ via the characterization of these $H^1$
spaces in terms of the square function and the fact that the
multiparameter square function acts iteratively when viewed as
a vector-valued operator. Using the fact that the dual of
$H^1(\R\otimes\R)$ is $\bmo(\R\otimes\R)$, by~\cite{CF}, and
likewise the dual of $H^1_{d,d}(\R\otimes\R)$ is
$\bmo_{d,d}(\R\otimes\R)$, by~\cite{Be}, it follows that
$\bmo(\R\otimes\R) \subset \bmo_{d,d}(\R\otimes\R)$.

The same argument shows that $\bmo(\R\otimes\R)$ is contained
in each of $\bmo_{d,\delta}(\R\otimes\R)$,
$\bmo_{\delta,d}(\R\otimes\R)$, and
$\bmo_{\delta,\delta}(\R\otimes\R)$.

Now we prove the other inclusion. Suppose $f\in \bmo_{d,d}(\R\otimes\R)
\cap \bmo_{d,\delta}(\R\otimes\R)\cap
\bmo_{\delta,d}(\R\otimes\R)\cap \bmo_{\delta,\delta}(\R\otimes\R)$.
We will show that $f\in \bmo(\R\otimes\R)$.

We must show that there is a positive constant $C$ such that
the inequality
\begin{eqnarray}\label{eqn:product BMO dyadic to BMO}
    \int\!\!\!\int_{T(\Omega)} \, |f * \psi_{y_1}\psi_{y_2}(t_1,t_2)|^2 \,
        \frac{dt_1 \, dy_1 \, dt_2 \, dy_2}{y_1y_2}
    \leq C|\Omega|
\end{eqnarray}
holds for all open sets~$\Omega$ with finite measure.

Fix such a set $\Omega\subset\R\otimes\R$. For each
point~$(t_1,y_1,t_2,y_2)$ in~$T(\Omega)$, by definition the two
intervals $I_{t_1,y_1}$ and $I_{t_2,y_2}$ satisfy
$I_{t_1,y_2}\times I_{t_2,y_2}\subset\Omega$. By
Proposition~\ref{prop:Mei}, for such $(t_1,y_1)$, we may choose
an interval $I_{t_1,y_1}^*$ such that $I_{t_1,y_1} \subset
I_{t_1,y_1}^*$, $|I_{t_1,y_1}^*| \leq C(\delta) |I_{t_1,y_1}|$,
and either $I_{t_1,y_1}^*\in\D$ or $I_{t_1,y_1}^*\in\Dd$.
Similarly, for such $(t_2,y_2)$, we may choose an interval
$I_{t_2,y_2}^*$ such that $I_{t_2,y_2} \subset I_{t_2,y_2}^*$,
$|I_{t_2,y_2}^*| \leq C(\delta) |I_{t_2,y_2}|$, and either
$I_{t_2,y_2}^*\in\D$ or $I_{t_2,y_2}^*\in\Dd$. Now, we let
\begin{eqnarray*}
    \F_1
    &:=& \{(t_1,y_1,t_2,y_2) \in T(\Omega) \mid I_{t_1,y_1}^*\in\D,\ \
    I_{t_2,y_2}^*\in\D\};\\
    \F_2
    &:=& \{(t_1,y_1,t_2,y_2) \in T(\Omega) \mid I_{t_1,y_1}^*\in\D,\ \
    I_{t_2,y_2}^*\in\Dd\};\\
    \F_3
    &:=& \{(t_1,y_1,t_2,y_2) \in T(\Omega) \mid I_{t_1,y_1}^*\in\Dd,\ \
    I_{t_2,y_2}^*\in\D\};\\
    \F_4
    &:=& \{(t_1,y_1,t_2,y_2) \in T(\Omega) \mid I_{t_1,y_1}^*\in\Dd,\ \
    I_{t_2,y_2}^*\in\Dd\}.
\end{eqnarray*}
Then $T(\Omega) = \F_1 \cup \F_2\cup \F_3\cup \F_4$, and the sets
$\F_i$ $(i = 1, 2, 3, 4)$ are pairwise disjoint. As a consequence,
\begin{eqnarray*}
    &&\int\!\!\!\int_{T(\Omega)} \, |f * \psi_{y_1}\psi_{y_2}(t_1,t_2)|^2 \,
        \frac{dt_1 \, dy_1 \, dt_2 \, dy_2}{y_1y_2}\\
    &&= \sum_{i=1}^4\int\!\!\!\int_{\F_i} \, |f * \psi_{y_1}\psi_{y_2}(t_1,t_2)|^2 \,
        \frac{dt_1 \, dy_1 \, dt_2 \, dy_2}{y_1y_2}\\
    &&=:(G_1)+(G_2)+(G_3)+(G_4).
\end{eqnarray*}

We first estimate $(G_1)$. For every $(t_1,y_1,t_2,y_2)$ in~$\F_1$,
we have $I_{t_1,y_1}^*\in \D$ and $I_{t_2,y_2}^*\in \D$. Let
$\widetilde{I}_{t_1,y_1}^*$ and $\widetilde{I}_{t_2,y_2}^*$ be
the parents of $I_{t_1,y_1}^*$ and $I_{t_2,y_2}^*$,
respectively. Define
\[
  \widetilde{\Omega}_1
  := \bigcup_{(t_1,y_1,t_2,y_2)\in\F_1}
    \widetilde{I}_{t_1,y_1}^*\times\widetilde{I}_{t_2,y_2}^*.
\]
Then, we can easily see that
\begin{eqnarray*}
    |\widetilde{\Omega}_1|&=&\Big|\bigcup_{(t_1,y_1,t_2,y_2)\in\F_1}
  \widetilde{I}_{t_1,y_1}^*\times\widetilde{I}_{t_2,y_2}^*\Big|=
  2^2 \Big|\bigcup_{(t_1,y_1,t_2,y_2)\in\F_1}
  I_{t_1,y_1}^*\times I_{t_2,y_2}^*\Big|\\
  &\leq& 2^2C(\delta)^2 \Big|\bigcup_{(t_1,y_1,t_2,y_2)\in\F_1}
  I_{t_1,y_1}\times I_{t_2,y_2}\Big|\leq 2^2C(\delta)^2 |\Omega|.
\end{eqnarray*}
Next, using the biparameter Haar expansion, we have
\[
  f * \psi_{y_1}\psi_{y_2}(t_1,t_2)= \sum_{R=I_1\times I_2\in\D\times\D}
  (f,h_R)h_{I_1}*\psi_{y_1}(t_1) h_{I_2}*\psi_{y_2}(t_2)
\]
for every $(t_1,y_1,t_2,y_2)\in\F_1$. We now claim that: If
$I_1\nsubseteq \widetilde{I}_{t_1,y_1}^*$, then
$h_{I_1}*\psi_{y_1}(t_1)=0$.

In fact, this claim follows from the analogous estimates in the
one-parameter case; see the estimates of $g_{12}$ in the proof of
Theorem \ref{thm:intersectionoftranslatesBMOR}. More precisely, we
explain it as follows.  First, from the properties of $h_{I_1}$ and
$\psi_{y_1}$, we see that if $I_1\cap
\widetilde{I}_{t_1,y_1}^*=\emptyset$, then
$h_{I_1}*\psi_{y_1}(t_1)=0$. Moreover, if $I_1\cap
\widetilde{I}_{t_1,y_1}^*\not=\emptyset$ and $I_1\nsubseteq
\widetilde{I}_{t_1,y_1}^*$, then $I_1$ must be larger than
$\widetilde{I}_{t_1,y_1}^*$ since both $I_1$ and
$\widetilde{I}_{t_1,y_1}^*$ are dyadic, which means that $I_1$ is
some ancestor of $ \widetilde{I}_{t_1,y_1}^*$. In this case, since
$\psi_{y_1}(t_1-\cdot)$ is supported in $I_{t_1,y_1}$ and $h_{I_1}$
is constant on $I_{t_1,y_1}$, we have $h_{I_1}*\psi_{y_1}(t_1)=0$.

Combining the two cases, we see that the claim holds.

Similarly, if $I_2\nsubseteq \widetilde{I}_{t_2,y_2}^*$,
then $h_{I_2}*\psi_{y_2}(t_2)=0$.

As a consequence, we have
\begin{eqnarray*}
  f * \psi_{y_1}\psi_{y_2}(t_1,t_2)&=& \sum_{R=I_1\times I_2\in\D\times\D}
  (f,h_R)h_{I_1}*\psi_{y_1}(t_1) h_{I_2}*\psi_{y_2}(t_2)\\
  &=& \sum_{R=I_1\times I_2\in\D\times\D, R\subset \widetilde{I}_{t_1,y_1}^*\times \widetilde{I}_{t_2,y_2}^*}
  (f,h_R)h_{I_1}*\psi_{y_1}(t_1) h_{I_2}*\psi_{y_2}(t_2).
\end{eqnarray*}

We now estimate $G_1$. First let $P_{\widetilde{\Omega}_1}f$ denote the projection
\[
  P_{\widetilde{\Omega}_1}f=\sum_{R=I_1\times I_2\in \widetilde{\Omega}_1}
  (f,h_R)h_{R}.
\]
From the results above, we have
\begin{eqnarray*}
    (G_1)&:=& \int\!\!\!\int_{\F_1} \, |f * \psi_{y_1}\psi_{y_2}(t_1,t_2)|^2 \,
        \frac{dt_1 \, dy_1 \, dt_2 \, dy_2}{y_1y_2}\\
    &=&\int\!\!\!\int_{\F_1} \, \bigg|\sum_{R=I_1\times I_2\in\D\times\D, R\subset \widetilde{I}_{t_1,y_1}^*\times \widetilde{I}_{t_2,y_2}^*}
  (f,h_R)h_{I_1}*\psi_{y_1}(t_1) h_{I_2}*\psi_{y_2}(t_2)\bigg|^2 \,
        \frac{dt_1 \, dy_1 \, dt_2 \, dy_2}{y_1y_2}\\
    &=&\int\!\!\!\int_{\F_1} \, \bigg|\sum_{R=I_1\times I_2\in\D\times\D, R\in \widetilde{\Omega}_1}
  (f,h_R)h_{I_1}*\psi_{y_1}(t_1) h_{I_2}*\psi_{y_2}(t_2)\bigg|^2 \,
        \frac{dt_1 \, dy_1 \, dt_2 \, dy_2}{y_1y_2}\\
    &=& \int\!\!\!\int_{\F_1} \, \bigg| P_{\widetilde{\Omega}_1}f *\psi_{y_1}(t_1) \psi_{y_2}(t_2)\bigg|^2 \,
        \frac{dt_1 \, dy_1 \, dt_2 \, dy_2}{y_1y_2},
\end{eqnarray*}
Here the last equality holds since the terms $R\in
\widetilde{\Omega}_1$  but $R\not\subset
\widetilde{I}_{t_1,y_1}^*\times \widetilde{I}_{t_2,y_2}^*$ are
zero.

Then, using the $L^2$ boundedness of the Littlewood--Paley $g$-function, we
see that
\begin{eqnarray*}
    G_1&\leq& \int\!\!\!\int_{\R_+^2\times \R_+^2} \,
    \bigg| P_{\widetilde{\Omega}_1}f*\psi_{y_1}(t_1) \psi_{y_2}(t_2)\bigg|^2 \,
        \frac{dt_1 \, dy_1 \, dt_2 \, dy_2}{y_1y_2}\\
       &\leq & C\|  P_{\widetilde{\Omega}_1}f\|_{L^2(\R\otimes\R)}^2\\
       &=& C\sum_{R\in \D\times\D}
       (P_{\widetilde{\Omega}_1}f,h_R)^2\\
       &=& C\sum_{R\in \widetilde{\Omega}_1}
       (f,h_R)^2\\
       &\leq& C|\widetilde{\Omega}_1|\|f\|_{\bmo_{d,d}(\R\otimes\R)}^2\\
       &\leq& 4C\cdot C(\delta)^2|\Omega|\|f\|_{\bmo_{d,d}(\R\otimes\R)}^2.
\end{eqnarray*}
Repeating the proof above, we find that $(G_2)\leq 4C\cdot
C(\delta)^2|\Omega|\|f\|_{\bmo_{d,\delta}(\R\otimes\R)}^2$,
$(G_3)\leq 4C\cdot
C(\delta)^2|\Omega|\|f\|_{\bmo_{\delta,d}(\R\otimes\R)}^2$ and
$(G_4)\leq 4C\cdot
C(\delta)^2|\Omega|\|f\|_{\bmo_{\delta,\delta}(\R\otimes\R)}^2$.
Combining the estimates from $G_1$ to $G_4$, we see that
inequality \eqref{eqn:product BMO dyadic to BMO} holds with a
constant~$C$ independent of~$\Om$, as required. In particular,
\begin{eqnarray*}
   \lefteqn{\|f\|_{\bmo(\R\otimes\R)}}\\
   &\leq& 2C^{1\over2} C(\delta) \max\{\|f\|_{\bmo_{d,d}(\R\otimes\R)},
    \|f\|_{\bmo_{d,\delta}(\R\otimes\R)}, \|f\|_{\bmo_{\delta,d}(\R\otimes\R)},
    \|f\|_{\bmo_{\delta,\delta}(\R\otimes\R)}\}.
    \qedhere
\end{eqnarray*}
\end{proof}


\section{$\vmo$ and product $\vmo$}
\label{sec:resultsproductVMO}
\setcounter{equation}{0}

We begin with the one-parameter case.

The space $\vmo$ of functions of vanishing mean oscillation was
introduced by Sarason in \cite{Sa} as the set of integrable
functions on the circle $\T$ satisfying
$\lim\limits_{\delta\rightarrow0} \sup\limits_{I:
|I|\leq\delta} \intav_I |f-f_I|dx = 0$. This space is the
closure in the \bmo\ norm of the subspace of $\bmo(\T)$
consisting of all uniformly continuous functions on $\T$.

An analogous space $\vmo(\R)$ on the real line was defined by
Coifman and Weiss \cite{CW}, where they proved that it is the
predual of the Hardy space $H^1(\R)$.
\begin{defn}[\cite{CW}]\label{def:VMO_closure}
    $\vmo(\R)$ is the closure of the space
    $C_0^\infty(\R)$ in the $\bmo(\R)$ norm~(\ref{def:BMO
    average}).
\end{defn}

An equivalent version of $\vmo(\R)$ can be defined as follows.
\begin{defn}\label{def:VMO_averages}
    The space $\vmo(\R)$ is the set of all functions
    $f\in\bmo(\R)$ satisfying the following conditions:
    \begin{enumerate}
        \item[(a)] $\lim\limits_{\delta\rightarrow 0}\
            \sup\limits_{Q:\ |Q|<\delta}\
            \intav_Q|f-f_Q|\,dx=0;$

        \item[(b)] $\lim\limits_{N\rightarrow \infty}\
            \sup\limits_{Q:\ |Q|>N}\
            \intav_Q|f-f_Q|\,dx=0;$ and

        \item[(c)] $\lim\limits_{R\rightarrow \infty}\
            \sup\limits_{Q:\ Q\cap B(0,R)\ =\ \emptyset\ }\
            \intav_Q|f-f_Q|\,dx=0$,
    \end{enumerate}
    where $Q$ denotes an interval in $\R$.
\end{defn}
For the proof of the equivalence of Definitions
\ref{def:VMO_closure} and \ref{def:VMO_averages} of $\vmo(\R)$,
see the Lemma in Section 3 of \cite[pp.166--167]{U}. See also
\cite[Theorem 7]{Bour}.

There is a third equivalent definition of $\vmo(\R)$, in terms of Carleson
measures.
\begin{defn}\label{def:VMO Carleson}
A function $f\in \bmo(\R)$ belongs to $\vmo(\R)$ if
\begin{enumerate}
    \item[(a)] $\lim\limits_{\delta\rightarrow 0}\
        \sup\limits_{Q:\ |Q|<\delta}\ {\displaystyle
        1\over\displaystyle  |Q|} \int_{T(Q)}
        |f*\psi_y(t)|^2\, {\displaystyle
        dt\,dy\over\displaystyle  y}=0;$

    \item[(b)] $\lim\limits_{N\rightarrow \infty}\
        \sup\limits_{Q:\ |Q|>N}\ {\displaystyle
        1\over\displaystyle  |Q|} \int_{T(Q)}
        |f*\psi_y(t)|^2\, {\displaystyle
        dt\,dy\over\displaystyle  y}=0;$ and

    \item[(c)] $\lim\limits_{R\rightarrow \infty}\
        \sup\limits_{Q:\ Q\cap B(0,R)\ =\ \emptyset\ }\
        {\displaystyle 1\over\displaystyle  |Q|}
        \int_{T(Q)} |f*\psi_y(t)|^2\, {\displaystyle
        dt\,dy\over\displaystyle  y}=0,$
\end{enumerate}
where $\psi$ is any function of the form specified in
Proposition \ref{prop:psiprops}.
\end{defn}

The equivalence of Definitions \ref{def:VMO_averages} and
\ref{def:VMO Carleson} can be shown as follows. First, it is a
routine estimate that
$$
   {1\over |Q|}\int_{T(Q)}|\psi_y*f(t)|^2\,{dt\,dy\over y}\leq
   C\intav_{4Q}|f(x)-f_{4Q}|^2\,dx,
$$
where $Q$ is an arbitrary interval in $\R$, $4Q$ is the
interval with the same midpoint as $Q$ and four times the
length, and $C$ is a constant independent of $Q$ and $f$. As a
consequence, (a), (b) and (c) in Definition~\ref{def:VMO
Carleson} follow directly from (a), (b) and (c) in
Definition~\ref{def:VMO_averages}. Conversely, suppose $f$
satisfies Definition~\ref{def:VMO Carleson}. Then it follows
from  Proposition 3.3 in \cite{DDSTY} that $f$ satisfies
Definition \ref{def:VMO_averages}. We note that \cite{DDSTY}
deals with the generalized space $\vmo_L(\R^n)$ of \vmo\
functions associated to a differential operator $L$ satisfying
the conditions that $L$ has a bounded holomorphic functional
calculus on $L^2(\R^n)$ and that the heat kernel of the
analytic semigroup generated by $L$ has suitable upper bounds.
We need only the special case when $L$ is the Laplacian
$\Delta$. It is shown in Proposition 3.6 in \cite{DDSTY}, by an
argument using the tent space corresponding to \vmo, that
$\vmo_{\Delta}(\R^n)$ coincides with the usual \vmo\ as in
Definition \ref{def:VMO_averages}.

We turn to the dyadic one-parameter case. Again we give three
equivalent definitions. First, $\vmo_d(\R)$ is the closure of
the space $C_0^\infty(\R)$ in the dyadic $\bmo_d(\R)$ norm of
formula (\ref{eqn:AverageBMOd}). Second, in terms of averages,
we define $\vmo_d(\R)$ as in Definition \ref{def:VMO_averages}
of $\vmo(\R)$ but taking the three suprema over only dyadic
intervals $I$ instead of arbitrary intervals~$Q$. The third
definition is in terms of a Carleson condition on Haar
coefficients as follows.

\begin{defn}\label{def:dyadic VMO}
A function $f\in \bmo_d(\R)$ belongs to the dyadic VMO space
$\vmo_d(\R)$ if
\begin{enumerate}
    \item[(a)] $\lim\limits_{\delta\rightarrow 0}\
        \sup\limits_{J:\ J\in\D}\
        {\displaystyle1\over\displaystyle |J|}\
        \sum\limits_{I:\ I\subset J,\ I\in\D,\ |I|<\delta\
        } (f,h_I)^2=0;$

    \item[(b)] $\lim\limits_{N\rightarrow \infty}\
        \sup\limits_{J:\ J\in\D}\
        {\displaystyle1\over\displaystyle |J|}\
        \sum\limits_{I:\ I\subset J,\ I\in\D,\ |I|>N\ }
        (f,h_I)^2=0;$ and

    \item[(c)] $\lim\limits_{R\rightarrow \infty}\
        \sup\limits_{J:\ J\in\D}\
        {\displaystyle1\over\displaystyle |J|}\
        \sum\limits_{I:\ I\subset J,\ I\in\D,\ I\cap
        B(0,R)\ =\ \emptyset\ } (f,h_I)^2=0.$
\end{enumerate}
\end{defn}
We note that, as for dyadic \bmo\  (Definition
\ref{def:CarlesonBMOd}), allowing $J$ in Definition
\ref{def:dyadic VMO} to range over all intervals, not just
dyadic intervals, produces the same space $\vmo_d(\R)$, with a
comparable norm.

\begin{prop}
    The following three definitions of the dyadic $\vmo$ space $\vmo_d(\R)$ are
    equivalent.
    \begin{enumerate}
        \item[\textup{(1)}] The definition of $\vmo_d(\R)$
            as the closure of~$C_0^\infty(\R)$ in the
            $\bmo_d(\R)$~norm
            \textup{(\ref{eqn:AverageBMOd})} in terms of
            average.

        \item[\textup{(2)}] The dyadic version
            Definition~\ref{def:VMO_averages}, in terms of
           averages.

        \item[\textup{(3)}] Definition~\ref{def:dyadic
            VMO}, in terms of Haar coefficients.
    \end{enumerate}
\end{prop}
\begin{proof}
The proof of the equivalence of definitions (1) and (2) follows
the corresponding proof in the continuous case. The equivalence
of definitions (2) and (3) follows directly from the estimates
(\ref{eqn:average to Carleson}) and (\ref{eqn:Carleson to
average}) in our proof of the  equivalence of Definitions
\ref{def:AverageBMOd} and \ref{def:CarlesonBMOd} of
$\bmo_d(\R)$.
\end{proof}

Similarly, for each $\delta\in\R$, there are three equivalent
definitions for the dyadic \vmo\ space $\vmo_\delta(\R)$ with
respect to the collection $\Dd$ of translated dyadic intervals,
which is defined at the start of
Section~\ref{sec:dyadicintervals}.


Next we  consider the product \vmo\ space $\vmo(\R\otimes\R)$,
for simplicity with only two parameters. Here we give only two
equivalent definitions, since the one-parameter definition in
terms of averages does not generalize naturally.

First, $\vmo(\R\otimes\R)$ is the closure of
$C_0^\infty(\R\otimes\R)$ in the product $\bmo(\R\otimes\R)$
norm. The second definition is in terms of Carleson measures,
as follows.
\begin{defn}\label{def:product VMO Carleson}
A function $f\in \bmo(\R\otimes\R)$ belongs to
$\vmo(\R\otimes\R)$ if
\begin{enumerate}
\item[\textup{(a)}] $\lim\limits_{\delta\rightarrow 0}\
    \sup\limits_{\Omega} {\displaystyle1\over
    \displaystyle|\Omega|}\ \sum\limits_{R\in\D\times\D:\
    R\subset \Omega,\ |R|<\delta\ }\ \int_{T(R)}
    |f*\psi_y(t)|^2\, {\displaystyle
    dt_1\,dy_1\,dt_2\,dy_2\over\displaystyle
    y_1y_2}=0;$

\item[\textup{(b)}] $\lim\limits_{N\rightarrow \infty}\
    \sup\limits_{\Omega} {\displaystyle1\over\displaystyle
    |\Omega|}\ \sum\limits_{R\in\D\times\D:\ R\subset
    \Omega,\ |R|>N\ }\ \int_{T(R)} |f*\psi_y(t)|^2\,
    {\displaystyle dt_1\,dy_1\,dt_2\,dy_2\over\displaystyle
    y_1y_2}=0;$ and

\item[\textup{(c)}] $\lim\limits_{R\rightarrow \infty}\
    \sup\limits_{\Omega} {\displaystyle1\over\displaystyle
    |\Omega|}\ \sum\limits_{R\in\D\times\D:\ R\subset
    \Omega,\ R\not\subset B(0,N)\ } \int_{T(R)}
    |f*\psi_y(t)|^2\, {\displaystyle
    dt_1\,dy_1\,dt_2\,dy_2\over\displaystyle y_1y_2}=0.$
\end{enumerate}
\end{defn}

Here and in the definitions below, $\Omega$ ranges over all
open sets in $\R\otimes\R$ of finite measure.

A short calculation shows that Definition \ref{def:product VMO
Carleson} is equivalent to the definition of
$\vmo(\R\otimes\R)$ given in \cite[Prop 5.1(ii)]{LTW}. In
\cite{LTW} the equivalence of this last definition and the
definition in terms of $C_0^\infty(\R\otimes\R)$ is proved.

Finally, we define the dyadic product \vmo\ space
$\vmo_{d,d}(\R\otimes\R)$, in two ways. First,
$\vmo_{d,d}(\R\otimes\R)$ is the closure of
$C_0^\infty(\R\otimes\R)$ in the dyadic
$\bmo_{d,d}(\R\otimes\R)$ norm. The second definition is in
terms of a Carleson condition on the Haar coefficients, as
follows.

\begin{defn}\label{def:product dyadic VMO}
A function $f\in \bmo_{d,d}(\R\otimes\R)$ belongs to the
\emph{dyadic product \vmo\ space} $\vmo_{d,d}(\R\otimes\R)$ if
\begin{enumerate}
    \item[(a)] $\lim\limits_{\delta\rightarrow 0}\
        \sup\limits_{\Omega}
        {\displaystyle1\over\displaystyle |\Omega|}\
        \sum\limits_{R\in\D\times\D:\ R\subset \Omega,\
        |R|<\delta\ } (f,h_R)^2=0;$

    \item[(b)] $\lim\limits_{N\rightarrow \infty}\
        \sup\limits_{\Omega}
        {\displaystyle1\over\displaystyle |\Omega|}\
        \sum\limits_{R\in\D\times\D:\ R\subset \Omega,\
        |R|>N\ } (f,h_R)^2=0;$ and

    \item[(c)] $\lim\limits_{N\rightarrow \infty}\
        \sup\limits_{\Omega}
        {\displaystyle1\over\displaystyle |\Omega|}\
        \sum\limits_{R\in\D\times\D:\ R\subset \Omega,\
        R\not\subset B(0,N)\ } (f,h_R)^2=0.$
\end{enumerate}
\end{defn}

We define $\vmo_{d,\delta}(\R\otimes\R)$,
$\vmo_{\delta,d}(\R\otimes\R)$, and
$\vmo_{\delta,\delta}(\R\otimes\R)$ similarly.

\begin{prop}\label{prop:dyadicproductVMOequiv}
    The following two definitions of dyadic product
    $\vmo_{d,d}(\R\otimes\R)$ are equivalent.
    \begin{enumerate}
        \item[\textup{(1)}] The definition of
            $\vmo_{d,d}(\R\otimes\R)$ as the closure
            of~$C_0^\infty(\R\otimes\R)$ in the
            $\bmo_{d,d}(\R\otimes\R)$~norm.

        \item[\textup{(2)}] Definition~\ref{def:product
            dyadic VMO}, in terms of Haar coefficients.
    \end{enumerate}
    The corresponding result holds for each of
    $\vmo_{d,\delta}(\R\otimes\R)$,
    $\vmo_{\delta,d}(\R\otimes\R)$, and
    $\vmo_{\delta,\delta}(\R\otimes\R)$.
\end{prop}

\begin{proof}
To prove this proposition, we follow the ideas given
in~\cite{LTW} for the continuous case. Denote by FH the linear
space of finite linear combinations of the Haar basis $\{h_R:\
R\in\D\times\D\}$.

We first claim that
\[
    \textup{clos}_{\bmo_{d,d}}\textup{FH}
    = \vmo_{d,d}.
\]
In fact, from Definition~\ref{def:product dyadic VMO}, it is
immediate that every Haar function $h_R$ belongs to
$\vmo_{d,d}$. Conversely, for each $f\in \vmo_{d,d}$, set
\begin{eqnarray}\label{eqn:FH approx}
   f_n
   := \sum_{R\in\D\times\D:\ R\subset B(0,2^n),\
     2^{-n}\leq |R|\leq 2^n} (f,h_R)h_R
\end{eqnarray}
for every positive integer $n$. Then it is clear that $f_n\in$
FH for each $n$. Moreover, $\|f - f_n\|_{\bmo_{d,d}}$ goes to~0
as $n$ tends to infinity, by conditions (a), (b) and (c) in
Definition~\ref{def:product dyadic VMO}. Hence the claim holds.

Next, we claim that
$$
  \textup{clos}_{\bmo_{d,d}}C_0^\infty
  =   \textup{clos}_{\bmo_{d,d}}\textup{FH}.
$$
In fact, we can see that $ C_0^\infty\subset
\textup{clos}_{\bmo_{d,d}}\textup{FH}$ since for every
$\varphi\in C_0^\infty$, we can verify that $\varphi$ satisfies
conditions (a), (b) and (c) in Definition \ref{def:product
dyadic VMO}. Hence by taking $\varphi_n$ as in
equation~\eqref{eqn:FH approx}, we can approximate $\varphi$ by
functions in~FH. Conversely, it is easy to verify that $
\textup{FH} \subset \textup{clos}_{\bmo_{d,d}}C_0^\infty$.

The proof of Proposition \ref{prop:dyadicproductVMOequiv} is complete.
\end{proof}

\begin{thm}\label{thm:intersectionoftranslates VMO}
    Suppose $\delta\in\R$ satisfies
    condition~\eqref{eqn:meicondR}.
    Then in the one-parameter case,
    \[
        \vmo(\R)
        = \vmo_{d}(\R) \cap
            \vmo_{\delta}(\R),
    \]
    and in the multiparameter case (stated for two parameters for
    simplicity),
    \[
        \vmo(\R\otimes\R)
        = \vmo_{d,d}(\R\otimes\R) \cap
            \vmo_{d,\delta}(\R\otimes\R) \cap
            \vmo_{\delta,d}(\R\otimes\R) \cap
            \vmo_{\delta,\delta}(\R\otimes\R).
    \]
\end{thm}

\begin{proof}
We first prove the one-parameter case.

The inclusion $\vmo(\R) \subset \vmo_{d}(\R) \cap
\vmo_{\delta}(\R)$ follows directly from the definitions of
$\vmo(\R)$,  $\vmo_d(\R)$, and  $\vmo_\delta(\R)$ via
averaging.

The proof of the other inclusion $\vmo(\R) \supset \vmo_{d}(\R)
\cap \vmo_{\delta}(\R)$ involves only minor modifications of
our proof for $\bmo(\R)$ above. We use the definition of
$\vmo_d(\R)$ and $\vmo_\delta(\R)$ in terms of Haar
coefficients (Definition~\ref{def:dyadic VMO}). The key point
is that the constant~$C$ in
inequality~\eqref{eqn:BMOlastinequality} is replaced by the
$\e$ from Definition~\ref{def:dyadic VMO}. We omit the details.

For the case of product~$\vmo$, we first show that
$\vmo(\R\otimes\R) \subset \vmo_{d,d}(\R\otimes\R)$. Take $f\in
\vmo(\R\otimes\R)$. Then $f$ is the limit in the
$\bmo(\R\otimes\R)$ norm of a sequence of functions $f_n$ in
$C_0^\infty(\R\otimes\R)$. Then $\{f_n\} \subset
\bmo(\R\otimes\R) \subset \bmo_{d,d}(\R\otimes\R)$, and also
$f_n$ converges to~$f$ in the $\bmo_d(\R\otimes\R)$ norm.
Therefore $f$ belongs to $\vmo_{d,d}(\R\otimes\R)$, as
required. The same argument shows that $f$ belongs to each of
$\vmo_{d,\delta}(\R\otimes\R)$, $\vmo_{\delta,d}(\R\otimes\R)$,
and $\vmo_{\delta,\delta}(\R\otimes\R)$.

Again, we can prove the other inclusion via minor modifications
of our proof for $\bmo(\R\otimes\R)$ above, using the
definition of our dyadic product $\vmo$ spaces in terms of Haar
coefficients (Definition~\ref{def:product dyadic VMO}). Again,
we omit the details.
\end{proof}

\section{Hardy spaces and maximal functions}
\label{sec:hardymaximal}
\setcounter{equation}{0}


We begin with the one-parameter case. Denote by $H^1$ the classical Hardy space and denote by $H^1_d$
(resp.~$H^1_{\delta}$) the dyadic Hardy space with respect to $\D$
(resp.~$\D^\delta$).
Also, denote by $M(f)$ the classical Hardy--Littlewood maximal
function, and denote by $M_d(f)$ (resp.~$M_{\delta}(f)$) the dyadic
Hardy--Littlewood maximal function with respect to $\D$
(resp.~$\D^\delta$).

Then we have the following results.
\begin{prop}\label{prop:H-L function and Hardy space}
Suppose $\delta\in(0,1)$ is far from dyadic rationals:
$d(\delta)>0$. Then the following relations hold between the
continuous and dyadic versions.
\begin{enumerate}
\item[\textup{(i)}] $H^1 = H^1_{d}+H^1_{\delta}$ with
    equivalent norms.
\item[\textup{(ii)}] For each $f\in L^1_\textup{loc}$, $M(f)$ is comparable with $M_{d}(f) +
    M_{\delta}(f)$ pointwise, and the
    implicit constants are independent of $f$.
\end{enumerate}
\end{prop}

\begin{proof}
For part (i), see Corollary~2.4 of \cite{Mei}.

For part (ii), it is immediate from the definitions  that for $f\in
L^1_\textup{loc}$, $M_{d}(f) \leq M(f)$ and $M_{\delta}(f)\leq
M(f)$. Thus, $M_{d}(f) + M_{\delta}(f) \leq 2M(f)$. Next, for each
interval $Q\subset\R$, by Proposition~\ref{prop:Mei} there is a
suitable interval $I\in\D$ or $I\in\Dd$ such that
  \[
     \frac{1}{|Q|}\int_Q|f|
     \leq C(\delta) \frac{1}{|I|}\int_I|f|.
  \]
As a consequence, we have $M(f) \leq
C(\delta)\max\{M_{d}(f),M_{\delta}(f)\} $.
\end{proof}

Now we give a corollary of
Theorem~\ref{thm:intersectionoftranslatesApRHpdbl} and
Proposition~\ref{prop:H-L function and Hardy space}, for weighted
Hardy spaces and weighted maximal functions.

Suppose $\om$ is a doubling weight. Denote by $H^1(\om)$ the
weighted Hardy space. Also, suppose $\om_d$ is a dyadic doubling
weight (resp. $\om_\delta$ is a $\delta$-dyadic doubling weight),
denote by $H^1_d(\om_d)$ (resp.~$H^1_{\delta}(\om_{\delta})$) the
dyadic (resp.~$\delta$-dyadic) Hardy space with respect to $\D$
(resp.~$\D^\delta$).

Denote by $M_\om(f)$ the weighted Hardy--Littlewood maximal
function:
  \[
     M_\om(f)(x)
     := \sup_{Q\ni x}\frac{1}{\om(Q)} \int_Q|f(y)|\om(y)dy,
  \]
where the supremum is taken over all intervals in $\R$. Similarly,
for a dyadic doubling weight $\om_d$, we define the dyadic weighted
Hardy--Littlewood maximal function $M_{d,\om_d}(f)$, and for a
$\delta$-dyadic doubling weight $\om_\delta$, we define the
$\delta$-dyadic weighted Hardy--Littlewood maximal function
$M_{\delta,\om_{\delta}}(f)$; here the supremum is taken over only
the intervals $I\in\D$ (resp.~$I\in\Dd$).

Then we have the following generalizations of Proposition
\ref{prop:H-L function and Hardy space} to the weighted case.

\begin{cor}\label{cor:weighted H-L function and Hardy space}
Suppose $\delta\in(0,1)$ is far from dyadic rationals: $d(\delta)>0$.
Suppose $\om$ is a doubling weight. Then the following relations
hold.
\begin{enumerate}
\item[\textup{(i)}] $H^1(\om) = H^1_{d}(\om)
    + H^1_{\delta}(\om)$ with equivalent norms.
\item[\textup{(ii)}] For
    each $f\in L^1_\textup{loc}$, $M_\om(f)$
    is pointwise equivalent to $M_{d,\om}(f)+M_{\delta,\om}(f)$. Here the implicit
    constants are independent of $f$.
\end{enumerate}
\end{cor}
In particular, this corollary holds for $w\in A_p$, $1\leq
p\leq\infty$. 

\begin{proof}
We first prove (i).  We recall the definition of atoms in weighted
Hardy spaces. A function $a$ is called an atom of the Hardy space
$H^1(\om)$ if it satisfies
\begin{enumerate}
\item[(a)] $\supp a\subset Q$ for some interval $Q\subset\R$;
\item[(b)] $\|a\|_{L^2(\om)}\leq \om(Q)^{-1\slash 2}$; and
\item[(c)] $\int a(x)\om(x)\,dx=0 $.
\end{enumerate}
Similarly, the atoms of $H^1_d(\om)$ (resp.~$H^1_{\delta}(\om)$)
satisfies the same conditions (a), (b) and (c) above with the extra
condition that $Q\in\D$ (resp.~$Q\in\D^\delta$).

From the definitions of atoms, it is immediate that each
$H^1_d(\om)$-atom or $H^1_{\delta}(\om)$-atom is an $H^1(\om)$-atom.
Thus, $H^1_d(\om)\subset H^1(\om)$ and $H^1_{\delta}(\om)\subset
H^1(\om)$ with norms $\|f\|_{H^1(\om)}\leq \|f\|_{H^1_d(\om)}$ and
$\|f\|_{H^1(\om)}\leq \|f\|_{H^1_{\delta}(\om)}$.

We now prove the converse. Suppose $a$ is an $H^1(\om)$-atom
satisfying (a), (b) and (c) with an interval $Q$. Then, by
Proposition \ref{prop:Mei}, there exists an interval $I$ such that
$Q\subset I$, $|I|\leq C(\delta)|Q|$ and $I\in \D$ or $I\in
\D^\delta$. Moreover, Proposition \ref{prop:comparable} implies that
$$
   \|a\|_{L^2(\om)}
   \leq \om(Q)^{-1\slash 2}
   \leq C(\delta)^{1\slash2}
    (C_\textup{dy})^{{1\over2}\log_2(4C(\delta))}\om(I)^{-1\slash2}.
$$
Let
$C_0:=C(\delta)^{1\slash2}(C_\textup{dy})^{{1\over2}\log_2(4C(\delta))}$.
Then the above inequality implies that $C_0^{-1}a$ is an
$H^1_d(\om)$-atom if $I\in\D$, and an $H^1_{\delta}(\om)$-atom if
$I\in\D^\delta$. Hence $H^1_d(\om)+H^1_{\delta}(\om)\subset
H^1_d(\om)$ with norms $
\|f\|_{H^1_d(\om)}+\|f\|_{H^1_{\delta}(\om)}\leq
C_0\|f\|_{H^1(\om)}$.

Now we turn to (ii). From the definition of the weighted
Hardy--Littlewood maximal functions, it is immediate to see that $
M_{d,\om}(f)\leq M_{\om}(f) $ and $M_{\delta,\om}(f)\leq M_{\om}(f)$
for every $f\in L^1_\textup{loc}$. Conversely, for each interval
$Q\subset \R$, by Proposition \ref{prop:Mei}, there exists an
interval $I$ such that $Q\subset I$, $|I|\leq C(\delta)|Q|$ and
$I\in \D$ or $I\in \D^\delta$. Moreover, from Proposition
\ref{prop:comparable}, we have
$$
   \om(I)
   \leq (C_\textup{dy})^{\log_2(4C(\delta))}\om(Q) |I|\slash|Q|
   \leq C(\delta)(C_\textup{dy})^{\log_2(4C(\delta))}\om(Q).
$$
Consequently, we obtain that
$$
  {1\over\om(Q)}\int_Q|f(y)|\om(y)\,dy
  \leq  C(\delta)(C_\textup{dy})^{\log_2(4C(\delta))}
    {1\over\om(I)}\int_I|f(y)|\om(y)\,dy,
$$
which implies that $M_{\om}(f) \leq
C(\delta)(C_\textup{dy})^{\log_2(4C(\delta))}
\left(M_{\delta,\om}(f)+M_{d,\om}(f)\right)$.
\end{proof}

Now we turn to the multiparameter case (stated for two parameters).

Denote by $H^1(\R\otimes\R)$ the product Hardy space. S.-Y. Chang
and R. Fefferman \cite{CF} showed that the dual of
$H^1(\R\otimes\R)$  is the product \bmo\ space $\bmo(\R\otimes\R)$
as mentioned in Section \ref{sec:resultsBMO}. Recently, Lacey,
Terwilleger and Wick showed that the predual of $H^1(\R\otimes\R)$
is the product \vmo\ space $\vmo(\R\otimes\R)$ as mentioned in
Section \ref{sec:resultsproductVMO}.

Next, as mentioned in the proof of Theorem
\ref{thm:intersectionoftranslatesBMORn}, we denote by
$H^1_{d,d}(\R\otimes\R)$ the dyadic product Hardy space with respect
to the dyadic rectangles $R\in\D\times\D$, whose norm is defined as
$$
   \|f\|_{H^1_{d,d}(\R\otimes\R)}
   := \Big\| \Big( \sum_{R\in \D\times\D} (f,h_R)^2 |R|^{-1} \chi_R \Big)^{1/2} \Big\|_1,
$$
where $\chi_R$ is the characteristic function of~$R$.  For more
information on the dyadic Hardy space, we refer to \cite{T}.
Similarly we can define the dyadic product Hardy spaces
$H^1_{d,\delta}(\R\otimes\R)$, $H^1_{\delta,d}(\R\otimes\R)$ and
$H^1_{\delta,\delta}(\R\otimes\R)$. It is  known that the dual
of $H^1_{d,d}(\R\otimes\R)$ is $\bmo_{d,d}(\R\otimes\R)$. We point
out that a direct proof can be found in \cite[Thm 4.2]{HLL} where
they deal with a more general setting of product sequence spaces.
Similarly the dual spaces of
$H^1_{d,\delta}(\R\otimes\R)$, $H^1_{\delta,d}(\R\otimes\R)$ and
$H^1_{\delta,\delta}(\R\otimes\R)$ are the dyadic product \bmo\
spaces $\bmo_{d,\delta}(\R\otimes\R)$,
$\bmo_{\delta,d}(\R\otimes\R)$ and
$\bmo_{\delta,\delta}(\R\otimes\R)$, respectively.

Now we address the duality of $\vmo_{d,d}(\R\otimes\R)$ with
$H^1_{d,d}(\R\otimes\R)$. In fact, the proof is similar to the proof of
the continuous version ($(\vmo(\R\otimes\R))^*=H^1(\R\otimes\R)$) as
shown in \cite{LTW}, where they relied on the facts that
$(H^1(\R\otimes\R))^*=\bmo(\R\otimes\R)$ and
$\textup{clos}_{H^1}\textup{FW}=H^1$. Here FW means the linear space
of finite linear combinations of product wavelets. Correspondingly,
we have the facts that
$(H^1_{d,d}(\R\otimes\R))^*=\bmo_{d,d}(\R\otimes\R)$ and that
$\textup{clos}_{H^1_{d,d}}\textup{FH}=H^1_{d,d}$, where the latter follows from
the definition of the norm of $H^1_{d,d}(\R\otimes\R)$.

Thus, the equality
$(\vmo_{d,d}(\R\otimes\R))^*=H^1_{d,d}(\R\otimes\R)$ holds. Similar
results hold for $\vmo_{d,\delta}(\R\otimes\R)$,
$\vmo_{\delta,d}(\R\otimes\R)$ and
$\vmo_{\delta,\delta}(\R\otimes\R)$.

\begin{prop}\label{prop:sum_of_translates_productH1}
    Suppose $\delta\in\R$ satisfies
    condition~\eqref{eqn:meicondR}.
    Then
    \[
        H^1(\R\otimes\R)
        = H^1_{d,d}(\R\otimes\R) +
            H^1_{d,\delta}(\R\otimes\R) +
            H^1_{\delta,d}(\R\otimes\R) +
            H^1_{\delta,\delta}(\R\otimes\R),
    \]
    with equivalent norms.
\end{prop}

This proposition follows from Theorem
\ref{thm:intersectionoftranslates VMO} by duality.

Next we turn to the maximal functions. Instead of the
Hardy--Littlewood maximal function, in the multiparameter setting
we consider the  strong maximal function $M_s$, which is defined as
follows. For a locally integrable function $f\in\R^2$,
\begin{eqnarray}\label{eqn:strong maximal}
   M_s(f)(x,y):=\sup_{R\ni(x,y)} {1\over|R|}\int_R|f(u,v)|\,du\,dv,
\end{eqnarray}
where the supremum is taken over all rectangles $R\in \R^2$.

Next, we denote by $M_{s}^{d,d}(f)$ the dyadic strong maximal
function defined by restricting the supremum in formula
(\ref{eqn:strong maximal}) to dyadic rectangles
$\R\in\D\times\D$. Also denote  by $M_{s}^{d,\delta}(f)$ the
dyadic strong maximal function where in (\ref{eqn:strong
maximal}) we take the supremum over all $R\in\D\times\Dd$. We
define $M_{s}^{\delta,d}(f)$ and $M_{s}^{\delta,\delta}(f)$
similarly. Then we have the following result.
\begin{prop}\label{prop:sum_of_translates_strong maximal}
    Suppose $\delta\in\R$ satisfies
    condition~\eqref{eqn:meicondR}.
    Then for each $f\in L^1_\textup{loc}(\R^2)$, $M_{s}(f)$ is
    comparable with
    $M_{s}^{d,d}(f)+M_{s}^{d,\delta}(f)+M_{s}^{\delta,d}(f)+M_{s}^{\delta,\delta}(f)$
    pointwise, and the implicit constants are independent of $f$.
\end{prop}
\begin{proof}
The proof of this proposition is similar to that of
Proposition \ref{prop:H-L function and Hardy space}(ii).
\end{proof}

In parallel with the one-parameter case, we define weighted
product Hardy spaces $H^1(\om)$, $H^1_{d,d}(\om_{d,d})$,
$H^1_{d,\delta}(\om_{d,\delta})$, $H^1_{\delta,d}(\om_{\delta,d})$,
and $H^1_{\delta,\delta}(\om_{\delta,\delta})$ for  doubling
weight $\om$ and dyadic doubling weights $\om_{d,d}$,
$\om_{d,\delta}$, $\om_{\delta,d}$ and $\om_{\delta,\delta}$. Also
we have the weighted strong maximal functions $M_{s,\om}$, $M_{s}^{d,d,\om_{d,d}}$,
$M_{s}^{d,\delta,\om_{d,\delta}}$, $M_{s}^{\delta,d,\om_{\delta,d}}$ and
$M_{s}^{\delta,\delta,\om_{\delta,\delta}}$.

Then, in parallel with  Corollary \ref{cor:weighted H-L function and
Hardy space}, we have the weighted version of Propositions
\ref{prop:sum_of_translates_productH1} and
\ref{prop:sum_of_translates_strong maximal}. We state it as follows, omitting the proof.
\begin{cor}\label{cor:product weighted H-L function and Hardy space}
Suppose $\delta\in(0,1)$ is far from dyadic rationals: $d(\delta)>0$.
Suppose $\om$ is a product doubling weight. Then the following
relations hold.
\begin{enumerate}
\item[\textup{(i)}] $H^1(\om) = H^1_{d,d}(\om)
    + H^1_{d,\delta}(\om)+ H^1_{\delta,d}(\om)+ H^1_{\delta,\delta}(\om)$ with equivalent norms.
\item[\textup{(ii)}] For each $f\in L^1_\textup{loc}$,
    $M_{s,\om}$ is pointwise equivalent to
    $M_{s}^{d,d,\om}(f)+M_{s}^{d,\delta,\om}(f)
    +M_{s}^{\delta,d,\om}(f)+M_{s}^{\delta,\delta,\om}(f)
    $. Here the implicit constants are independent of $f$.
\end{enumerate}
\end{cor}


\begin{thebibliography}{DDSTY}

\bibitem[Ber]{Be} %
{\sc A.~Bernard}, %
`Espaces $H^1$ de martingales \`a  deux indices. Dualit\'e
avec les martingales de type ``BMO'' (French)', %
\emph{Bull. Sci. Math. (2)}~103 (1979), no. 3, 297--303.

\bibitem[BR]{BR} %
{\sc O.~Beznosova and A.~Reznikov}, %
`Sharp estimates involving $A_\infty$ and $L\log L$
constants, and their applications to PDE', %
arXiv:1107.1885v1.

\bibitem[Bou]{Bour} %
{\sc G.~Bourdaud}, %
`Remarques sur certains sous-espaces de $\bmo(\R^n)$
et de bmo$(\R^n)$', %
\emph{Ann.\ Inst.\ Fourier(Grenoble)}~52 (2002), 1187--1218.

\bibitem[Buc]{B} %
{\sc S.M.~Buckley}, %
`Summation conditions on weights', %
\emph{Michigan Math.\ J.}~40 (1993), 153--170.

\bibitem[CF]{CF} %
{\sc S.-Y.A.~Chang and R.~Fefferman}, %
`A continuous version of duality of $H^1$ with $\bmo$ on the
bidisc', %
\emph{Annals of Math.}~112 (1980), no.~1, 179--201.

\bibitem[CW]{CW}
{\sc R.~Coifman and G.~Weiss}, %
`Extensions of Hardy spaces and their use in analysis', %
\emph{Bull. Am. Math. Soc.}~83 (1977), 569--645.

\bibitem[CN]{CN} %
{\sc D.~Cruz-Uribe and C.J.~Neugebauer}, %
`The structure of the reverse H\"older classes', %
\emph{Trans.\ Amer.\ Math.\ Soc.}~347 (1995), 2941--2960.

\bibitem[DDSTY]{DDSTY}%
{\sc D.G.~Deng, X.T.~Duong, L.~Song, C.Q.~Tan and L.X.~Yan},%
`Functions of vanishing mean oscillation associated with operators
and applications', %
\emph{Michigan\ Math.\ J.}~56 (2008), 529--550.

\bibitem[GCRF]{GCRF} %
{\sc J.~Garcia-Cuerva and J.L.~Rubio de Francia}, %
\emph{Weighted norm inequalities and related topics}, %
North-Holland, New York, NY, 1985. %

\bibitem[Gar]{Gar} %
{\sc J.B.~Garnett}, %
\emph{Bounded analytic functions}, %
Academic Press, Orlando, FL, 1981. %

\bibitem[GJ]{GJ} %
{\sc J.B.~Garnett and P.W.~Jones}, %
`$\bmo$ from dyadic $\bmo$', %
\emph{Pacific.\ J.\ Math.}~99 (1982), no.~2, 351--371. %

\bibitem[Gra]{Gra} %
{\sc L.~Grafakos}, %
\emph{Classical and modern Fourier analysis}, %
Pearson Education, Upper Saddle River, NJ, 2004. %

\bibitem[HP]{HP} %
{\sc  T. Hyt\"onen and C. P\'erez}, %
`Sharp weighted bounds involving $A_\infty$', %
arXiv:1103.5562v1. %

\bibitem[HLL]{HLL} {\sc Y. Han, J. Li and G. Lu},
    `Duality of multiparameter Hardy space $H^p$ on product
    spaces of homogeneous type', \emph{Ann. Scuola Norm. Sup. Pisa
    Cl. Sci. (5)}, Vol. IX (2010), 645--685.

\bibitem[Jaw]{Jaw} %
{\sc B.~Jawerth}, %
`Weighted inequalities for maximal operators: linearization,
localization and factorization', %
\emph{Amer.\ J.\ Math.}~108 (1986), 361--414. %

\bibitem[JoNi]{JoNi} %
{\sc F.~John and L.~Nirenberg}, %
`On functions of bounded mean oscillation', %
\emph{Comm.\ Pure Appl.\ Math.}~14 (1961), 415--426. %

\bibitem[JoNe]{JoNe} %
{\sc R.~Johnson and C.J.~Neugebauer}, %
`Homeomorphisms preserving~$A_p$', %
\emph{Rev.\ Mat.\ Iberoamericana}~3 (1987), 249--273. %


\bibitem[LTW]{LTW} %
{\sc M.T. Lacey, E. Terwilleger, and B. Wick}, %
`Remarks on product~\vmo', %
\emph{Proc.\ Amer.\ Math.\ Soc.}~134 (2006), no.~2, 465--474. %

\bibitem[Lin]{L}
{\sc K.-C. Lin}, %
\emph{Harmonic analysis on the bidisc}, %
Ph.D.\ thesis, UCLA, 1984.

\bibitem[Mei]{Mei} T. Mei, %
`$\bmo$ is the intersection of two translates of
dyadic~$\bmo$', %
\emph{C.~R.~Acad.\ Sci.\ Paris, Ser.~I}~336 (2003), 1003--1006.

\bibitem[M]{M} %
{\sc B.~Muckenhoupt}, %
`Weighted norm inequalities for the Hardy maximal function', %
\emph{Trans.\ Amer.\ Math.\ Soc.}~165 (1972), 207--226.


\bibitem[O]{O}%
{\sc K.~Okikiolu}, %
`Characterization of subsets of rectifiable curves in $\R^n$', %
\emph{J. London Math. Soc. (2)}~46 (1992), 336--348.

\bibitem[P]{P} %
{\sc J.~Pipher}, %
\emph{Double Index Square Functions and Bounded Mean Oscillation on the Bidisc}, %
Ph.D.\ thesis, UCLA, 1985.

\bibitem[PW]{PW} %
{\sc J.~Pipher and L.A.~Ward}, %
`$\bmo$ from dyadic $\bmo$ on the bidisc', %
\emph{J.\ London Math.\ Soc.}~77 (2008), 524--544. %

\bibitem[PWX]{PWX} %
{\sc J.~Pipher, L.A.~Ward and X.~Xiao}, %
`Geometric-arithmetic averaging of dyadic weights', %
\emph{Rev. Mat. Iberoamericana}~27 (2011), no.~3, 953--976. %

\bibitem[Sar]{Sa} %
{\sc D.~Sarason}, %
`Functions of vanishing mean oscillation', %
\emph{Trans. Amer. Math. Soc.}~207 (1975), 391---405. %

\bibitem[Ste]{S} %
{\sc E.M.~Stein}, %
\emph{Harmonic analysis: Real-variable methods, orthogonality,
and oscillatory integrals}, %
Princeton University Press, Princeton, NJ, 1993. %

\bibitem[T]{T} %
{\sc S.~Treil}, %
`$H^1$ and dyadic $H^1$', in \emph{Linear and Complex Analysis:
Dedicated to V. P. Havin on the Occasion of His 75th Birthday}
(ed. S.~Kislyakov, A.~Alexandrov, A.~Baranov), Advances in the
Mathematical Sciences~226 (2009), AMS, 179--194.

\bibitem[U]{U} %
{\sc A.~Uchiyama}, %
`On the compactness of operators of the Hankel type', %
\emph{T\^ohoku Math. J.}~30 (1978), 163---171. %

\bibitem[W]{W}%
{\sc L.A.~Ward}, %
`Translation averages of dyadic weights are not always good
weights', %
\emph{Rev.\ Mat.\ Iberoamericana}~18 (2002), 379--407. %

\end{thebibliography}
\end{document}